\documentclass[12pt]{article}
\usepackage[a4paper, total={17cm,25cm}]{geometry}

\usepackage{amsmath,amssymb,amsfonts}
\usepackage{color, graphics}
\usepackage{tikz}
\usetikzlibrary{arrows}

\allowdisplaybreaks

\usepackage{float}

\usepackage[shortlabels]{enumitem}

\graphicspath{ {./ImagesForFigures/} }

\usepackage[alphabetic,y2k,initials]{amsrefs}

\def\sign{\mathrm{sign\,}}

\newcommand{\Sn}[1]{\mathbf{S}^#1}

\def\sn{\,\mathrm{sn}}
\def\cn{\,\mathrm{cn}}
\def\dn{\,\mathrm{dn}}
\newcommand{\R}[1]{\mathbf{R}^{#1}}

\newcommand{\Mink}{\mathbf{M}^{3}}
\newcommand{\Hyp}{\mathcal{H}}

\newcommand{\ip}[2]{\left<  #1, #2 \right>}

\newcommand{\ep}{\varepsilon}

\numberwithin{equation}{section}

\newtheorem{theorem}{Theorem}[section]
\newtheorem{proposition}[theorem]{Proposition}
\newtheorem{corollary}[theorem]{Corollary}
\newtheorem{lemma}[theorem]{Lemma}
\newtheorem{remark}[theorem]{Remark}

\newtheorem{example}[theorem]{Example}
\newtheorem{definition}[theorem]{Definition}

\newcommand{\twobytwo}[4]{\left(\begin{array}{cc}
		#1 & #2  \\
		#3 & #4 \end{array} \right)}
\newcommand{\threebythree}[9]{\left( \begin{array}{ccc}
		#1 & #2 & #3 \\
		#4 & #5 & #6 \\
		#7 & #8 & #9 \end{array} \right)}
\newcommand{\fourbyfour}[9]{\left( \begin{array}{cccc}
		#1 & #2 & \cdots & #3 \\
		#4 & #5 & \cdots & #6\\
		\vdots & \vdots & \cdots  & \vdots \\
		#7 & #8 & \cdots & #9 \end{array} \right)}
\newcommand{\iltwobytwo}[4]{\left(\begin{smallmatrix} #1 & #2 \\ #3 & #4 \end{smallmatrix}\right)} 
\newcommand{\ilvector}[2]{(\begin{smallmatrix} #1 \\ #2 \end{smallmatrix})} 

\newenvironment{proof}[1][Proof]{\noindent\textit{#1.} }{\hfill$\Box$\medskip}

 \title{Integrable Billiards on a Minkowski Hyperboloid: Extremal Polynomials and Topology}

\usepackage{authblk}
\author[1,3]{Vladimir Dragovi\'c}
\author[2]{Sean Gasiorek}
\author[2,3]{Milena Radnovi\'c}
\affil[1]{\textsc{The University of Texas at Dallas, Department of Mathematical Sciences}}
\affil[2]{\textsc{The University of Sydney, School of Mathematics and Statistics}}
\affil[3]{\textsc{Mathematical Institute SANU, Belgrade}}
\affil[ ]{\texttt{vladimir.dragovic@utdallas.edu, sean.gasiorek@sydney.edu.au, milena.radnovic@sydney.edu.au}}

\date{}

\begin{document}

\maketitle

\

\centerline{
\emph{Dedicated to R.~Baxter on the occasion of his 80-th anniversary.}
}

\

\begin{abstract}
We consider billiard systems within compact domains bound\-ed by confocal conics on a hyperboloid of one sheet in the Minkowski space. We derive conditions for elliptic periodicity for such billiards. We describe the topology of those billiard systems in terms of Fomenko invariants. We provide then periodicity conditions in terms  of functional Pell equations and related extremal polynomials.
Several examples are computed in terms of elliptic functions and classical Chebyshev  and  Zolotarev polynomials, as extremal polynomials over one or two intervals. These results are contrasted with the cases of billiards in the Minkowski and the Euclidean planes. 

\smallskip

\emph{Keywords:} Billiards, Minkowski space, hyperboloid, confocal quadrics, periodic trajectories, Zolotarev  polynomials, Chebyshev polynomials, Fomenko invariants

\smallskip

\textbf{MSC2020:} 37C83, 37J70, 37J38,37J46, 41A50, 14H70 
\end{abstract}

\tableofcontents

\section{$2-2$-symmetric relations, elliptic functions, and elliptical billiards}\label{intro}

The last section of Baxter's celebrated book (\cite{Bax82}, p.~471) starts with:
\begin{quote}
	\emph{``In the Ising, eight-vertex and hard hexagon models we
		encounter symmetric biquadratic relations, of the form
		(\ref{eq:biquadratic})''}.
\end{quote}

\begin{equation}\label{eq:biquadratic}
E: au^2v^2+b(u^2v+uv^2)+c(u^2+v^2)+2duv+e(u+v)+f=0.
\end{equation}

In the sequel of the last section of \cite{Bax82}, Baxter derives an
elliptic parametrization of a symmetric biquadratic, providing an effective proof of the classical theorem of Euler, which denoted the beginning
of the study of elliptic functions and related addition theorems
(see \cite{Euler1766}).

\begin{theorem}[Euler theorem, 1766]
	For the general symmetric $2-2$-correspondence (\ref{eq:biquadratic})
	there exists an even elliptic function $\phi$ of the second degree
	and a constant shift $c$ such that
	$$
	u=\phi(z),\quad v=\phi(z\pm c).
	$$
\end{theorem}

Elliptic functions and their addition formulae play a prominent role in the entire Baxter opus \cites{Bax71a,Bax71b,Bax72b,Bax72a} and in the theory of integrable systems in general. In his book in particular they appear as soon as in the second paragraph of the preface. For further references we list well-known identities for the Jacobi elliptic functions (see eg. \cite{AK90}):
\begin{equation}\label{eq:ellipticidentities}
\begin{aligned}
\kappa^2\sn^2z+\dn^2z&=1,
\\
\sn(z+w)&=\frac{\sn\, z\cn\, w\dn \,w + \sn\, w\cn\, z\dn\, z}{1-\kappa^2\sn^2 z \sn ^2 w},
\\
\sn (K-z)&=\frac{\cn\, z}{\dn\, z},
\\
K&=\int_0^1\frac{dt}{\sqrt{(1-t^2)(1-\kappa^2t^2)}}.
\end{aligned}
\end{equation}
Here $\kappa$ is a constant different from $0, 1$.
By using the above argumentation, Baxter managed to get his
celebrated $R$--matrix, which is also known as
$XYZ$ $R$--matrix and the Eight Vertex Model $R$--matrix because of its
fundamental role in both of these very important models of quantum
and statistical mechanics, respectively. 

Symmetric biquadratic relations \eqref{eq:biquadratic} also play an important role in the Poncelet theorem and related questions of integrable billiards within conics. Let us start with the situation of the Poncelet theorem. Suppose
conics $\Gamma$ and $\mathcal{K}$ are given. Consider the $2-2$-correspondence on $\Gamma$ induced by $\mathcal{K}$ in the following
way. To a point $M\in\Gamma$, one can correspond points $M_1$ and $M_1^\prime$ on $\Gamma$, such
that the lines $L_{MM_1}$ and $L_{MM_1^\prime}$ are tangent to the conic
$\mathcal{K}$. In this way, a symmetric  $2-2$-correspondence is
defined. Moreover, every symmetric $2-2$-correspondence on a conic
is defined in this way. The Italian mathematician Trudi,  around 1853, studied  the Poncelet theorem  in terms of compositions of such symmetric $2-2$-relations and provided another proof in  \cites{Trud1853,Trud1863}.

Let us recall that the addition formulae for the Jacobi elliptic functions were used by Jacobi himself in his proof of the Poncelet theorem for circles. More about symmetric $2-2$-relations and their role in integrable systems can be found in \cites{V1991,V1992,Duist2010, DR2011}.

In the present paper we study a new instance of symmetric $2-2$-relations which appears in integrable billiard dynamics in the Minkowski space on a hyperboloid of one sheet.
Such billiards were recently introduced in \cite{GaR}, where
Poncelet-type theorems and corresponding analytic conditions for periodicity were derived.
In this work, we study the periodicity conditions for the dynamics in accordance with the general ideology from \cite{DR2019b}: we relate them to the extremal polynomials on the unions of two intervals. 
As it is known from classics, \cites{Zolotarev1877,AK90}, such polynomials are parametrized by elliptic functions. The identities and addition formulas for elliptic functions, like \eqref{eq:ellipticidentities}, will play significant role in parametrizing the periodic trajectories of this dynamical system.

This paper is organized as follows.
In Section \ref{HyperboloidIntro}, we provide necessary review of the confocal families and billiards on the one-sheeted hyperboloid in the Minkowski space. We conclude Section \ref{HyperboloidIntro} with new results on elliptic periodicity of such billiards, which we collect in Section \ref{sec:EllipticPeriodic}.
In Section \ref{TopProperties}, we describe the topological properties of the integrable billiards from the previous section in terms of Fomenko graphs. 
In Section \ref{sec:discriminantly-poly}, we show through examples that the analytic conditions for closed billiard trajectories lead to discriminantly factorizable polynomials, a generalized form of discriminantly separable polynomials.
In Section \ref{sec:extremal}, we conclude the paper by establishing a connection of those conditions with generalized extremal polynomials on two intervals, so-called Zolotarev polynomials, and deduce properties of corresponding rotation numbers.

\section{Confocal conics and billiards on the hyperboloid of one sheet}
\label{HyperboloidIntro}
In this section, we present main notions are results regarding confocal families of conics and corresponding billiards on the one-sheeted hyperboloid in the Minkowski space.

The \emph{three-dimensional Minkowski space} $\Mink$ is the real 3-dimensional vector space $\R{3}$ with the symmetric nondegenerate bilinear form
\begin{equation}\label{eq:form}
\ip{v}{w} = - x_v x_w + y_v y_w  +  z_v z_w.
\end{equation}

On the hyperboloid of one sheet
\begin{equation*}
\Hyp\ :\ -x^2 + y^2 + z^2 =1
\end{equation*}
in $\Mink$, the metric $ds^2 = -dx^2 + dy^2 + dz^2$ is a Lorentz metric of constant curvature.
Geodesics of this metric are the intersections of $\Hyp$ and planes containing the origin. 
We call these geodesics \emph{space-}, \emph{time-}, or \emph{light-like}, if their tangent vectors $v$ are such, i.e.~if $\ip{v}{v}$ is respectively positive, negative, or zero.
Note that the light-like geodesics on $\Hyp$ are exactly its generatrices.

\subsection{Conics on the hyperboloid}

A conic on $\Hyp$ is defined as the intersection of $\Hyp$ with the following cone:
\begin{equation}\label{eq:coneA}
-\frac{x^2}a+\frac{y^2}b+\frac{z^2}c=0.
\end{equation}
We will assume that the cone is not symmetric, i.e.~$b\neq c$.
Moreover, without loss of generality, we can then assume $b<c$.
Its intersection with $\Hyp$ bounds a compact domain on $\Hyp$ if and only if all generatrices of the cone are space-like, see \cite{GaR}.
That will happen exactly in one of the following two cases:
\begin{itemize}
	\item if $0<a<b<c$ then the cone \eqref{eq:coneA} divides $\Hyp$ into one compact domain and two unbounded domains.
	The conic is called \emph{collared $\Hyp$-ellipse} and consists of two components which are symmetric to each other with respect to the coordinate $yz$-plane, see the left side of Figure \ref{TwoCases};
	\item if $b<0<a<c$ then the cone divides $\Hyp$ into two compact domains and one unbounded domain. 
	The conic is called \emph{transverse $\Hyp$-ellipse} and consists of two components which are symmetric to each other with respect to the coordinate $xy$-plane, see the right side of Figure \ref{TwoCases}. 
\end{itemize}
\begin{figure}[h]
	\centering
	\begin{tabular}{c c}
	 \includegraphics[width=0.43\textwidth]{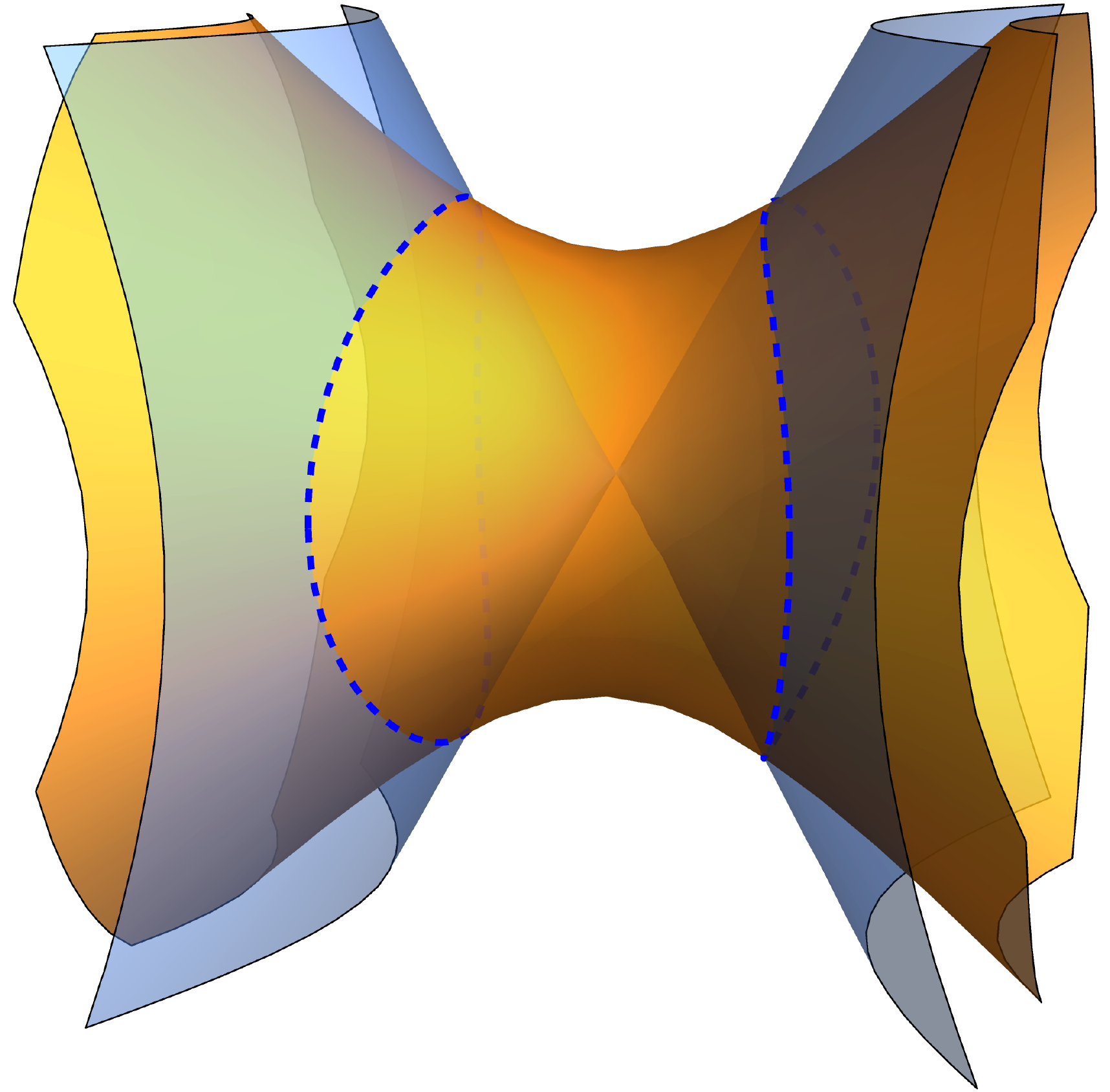} & \includegraphics[width=0.43\textwidth]{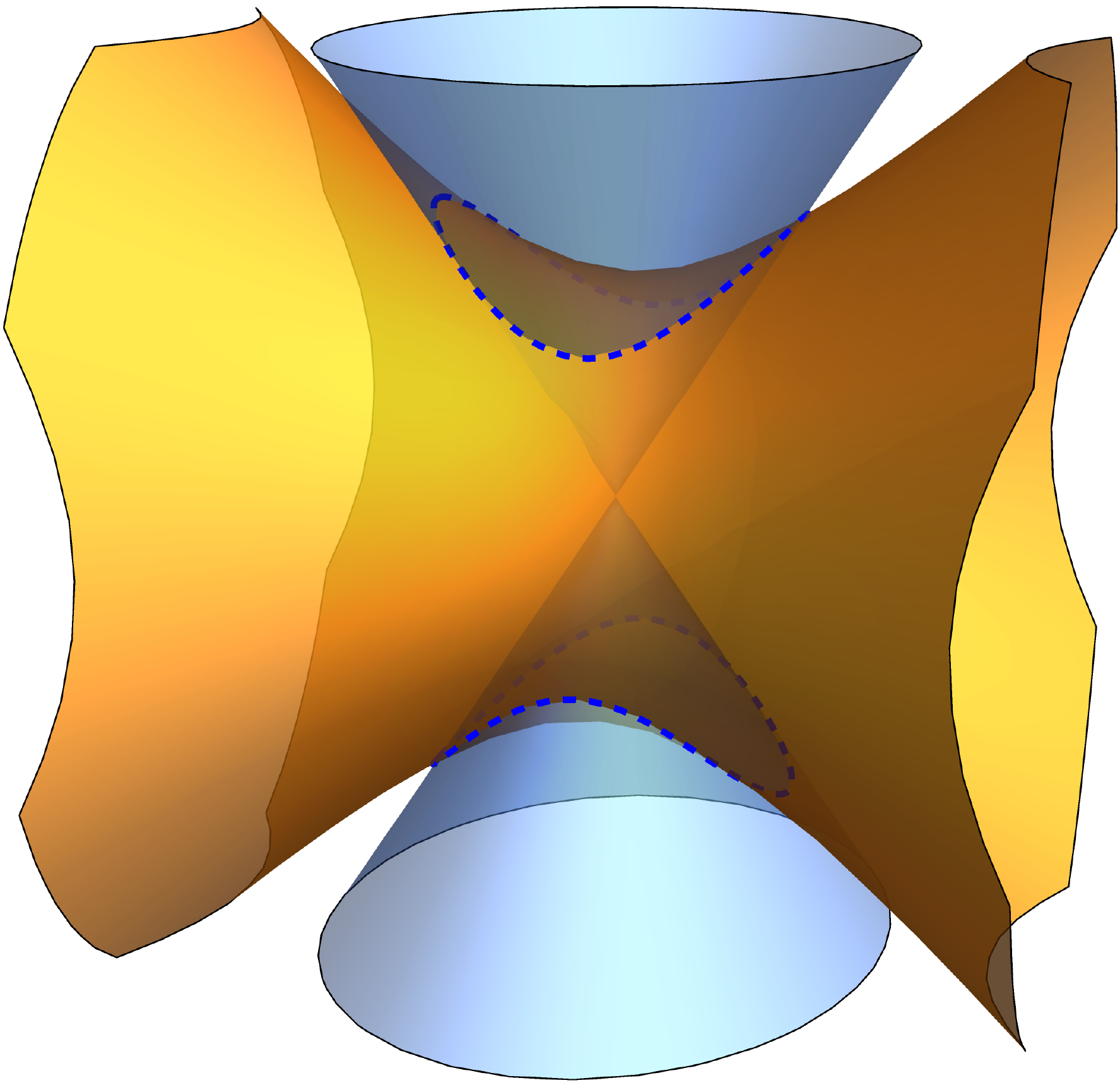}
	\end{tabular}
	\caption{Two geometric possibilities for the intersection of the cone and $\Hyp$ determining a compact domain: the collared (on the left) and transverse $\Hyp$-ellipse (on the right). }
	\label{TwoCases}
\end{figure}

\subsection{Confocal families}\label{sec:confocal}

The family of conics which are confocal to the conic given by \eqref{eq:coneA} on $\Hyp$ is given by: 
\begin{equation}
\mathcal{C}_{\lambda}
:
-\frac{x^2}{a-\lambda} + \frac{y^2}{b-\lambda} + \frac{z^2}{c-\lambda}=0.
\label{ConfFam1} 
\end{equation}
Note that \eqref{eq:coneA} is denoted by $\mathcal{C}_0$ here.
We describe the confocal family in more detail and provide illustrations in Figure \ref{fig:ConfocalCurves}. 

If $\mathcal{C}_0$ is a collared $\Hyp$-ellipse, i.e.~$0<a<b<c$, then the confocal family \eqref{ConfFam1} consists of two types of conics:
\begin{itemize}
\item Collared $\Hyp$-ellipses, corresponding to parameters $\lambda<a$. We note that those conics are space-like.
\item The conics of hyperbolic type, with $b<\lambda<c$. Such a conic always consists of $4$ components, symmetric to each other with respect to the coordinate $xy$- and $xz$-planes.
We note that those conics are time-like.
\end{itemize}
The family also contains the following degenerate conics:
\begin{itemize}
	\item The circle $\mathcal{C}_a$, which is the intersection of the hyperboloid with the coordinate $yz$-plane. The circle is space-like.
	\item The hyperbolae $\mathcal{C}_b$ and $\mathcal{C}_c$, which are respectively the intersections of the coordinate $xz$- and $xy$- planes  with $\Hyp$.
	Those curves are time-like.
\end{itemize}

If $\mathcal{C}_0$ is a transverse $\Hyp$-ellipse, i.e.~$b<0<a<c$, then the confocal family \eqref{ConfFam1} consists of four subfamilies of conics:
\begin{itemize}
	\item Transverse $\Hyp$-ellipses corresponding to parameters $b<\lambda<a$. $\mathcal{C}_0$ belongs here. All conics here consist of two closed connected components, which are symmetric to each other with respect to the coordinate $xy$-plane.
	Each component consists of two space-like and two time-like arcs.
	
	\item Transverse $\Hyp$-ellipses corresponding to parameters  $a<\lambda<c$. All conics here consist of two closed connected components, which are symmetric to each other with respect to the coordinate $xz$-plane. Each component here also consists of two space-like and two time-like arcs.

	\item The conics of hyperbolic type, with $\lambda>c$. Such a conic always consists of $4$ unbounded connected components, symmetric to each other with respect to the coordinate $xz$- and $yz$-planes. Each connected component consists of one bounded space-like arc and two unbounded time-like ones.

	\item The conics of hyperbolic type, with $\lambda<b$. As in the previous case such a conic always consists of $4$ unbounded connected components, but here they are symmetric to each other with respect to the coordinate $xy$- and $yz$-planes. Each connected component consists of one bounded space-like arc and two unbounded time-like ones.
\end{itemize}
We note that the conics from this confocal family have joint tangent lines, which are generatrices of $\Hyp$ touching $\mathcal{C}_0$.
There are $8$ such generatrices and their intersection points are the foci of the confocal family:
\begin{itemize}
	\item four foci in the coordinate $xy$-plane, denoted:
	$F_{\pm\pm}^z = \left(
	\pm\sqrt{\dfrac{c-a}{a-b}},
	\pm\sqrt{\dfrac{c-b}{a-b}},
	0
	\right)$;
	\item four foci in the coordinate $xz$-plane, denoted:
	$F_{\pm\pm}^y = \left(
	\pm\sqrt{\dfrac{a-b}{c-a}},
	0,
	\pm\sqrt{\dfrac{c-b}{c-a}}
	\right)$;
	\item four foci in the coordinate $yz$-plane, denoted:
	$F_{\pm\pm}^x = \left(
	0,
	\pm\sqrt{\dfrac{a-b}{c-b}},
	\pm\sqrt{\dfrac{c-a}{c-b}}
	\right)$.
\end{itemize}
Those $8$ generatrices divide $\Hyp$ into $20$ domains.
Twelve of those domains contain conics from the confocal family, while eight remaining domains contain none of the conics.

\begin{figure}
    \centering
    	\begin{tabular}{c c}	 \includegraphics[width=0.45\textwidth]{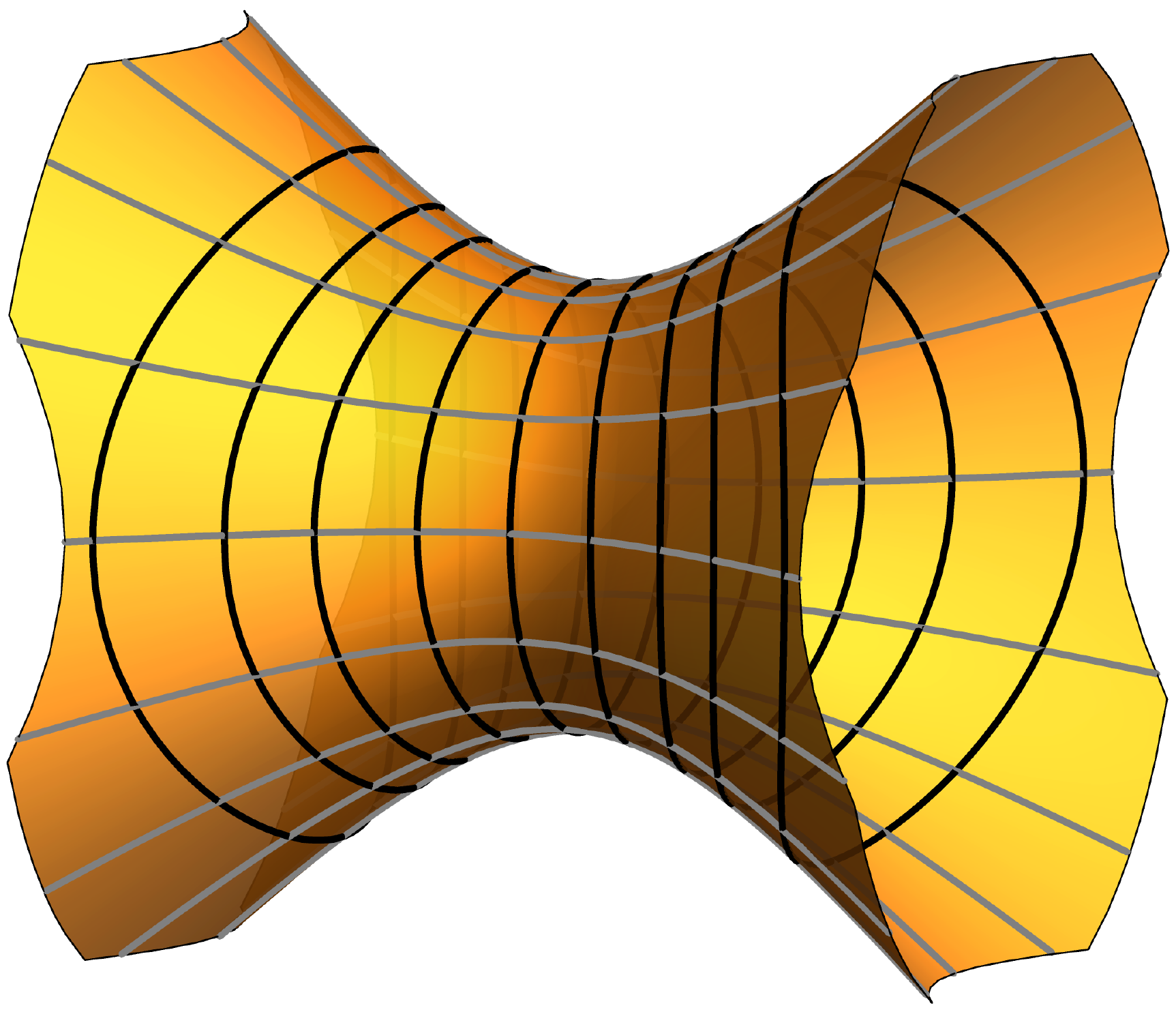} & \includegraphics[width=0.40\textwidth]{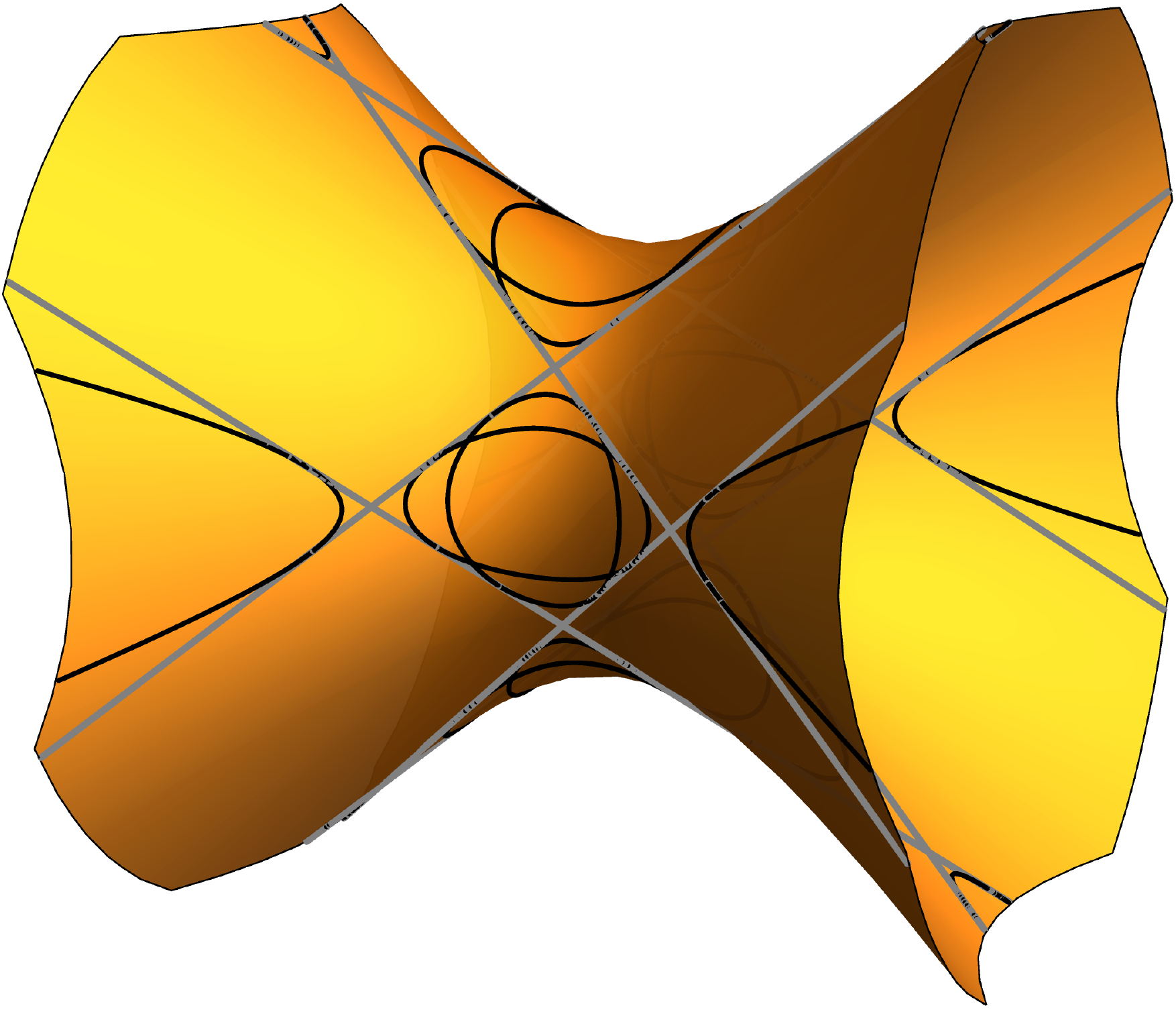}
	\end{tabular}
    \caption{Families of confocal curves in the collared (left) and transverse (right) $\Hyp$-ellipses. }
    \label{fig:ConfocalCurves}
\end{figure}

The family also contains degenerate conics $\mathcal{C}_a$, $\mathcal{C}_b$, $\mathcal{C}_c$, which are contained in the corresponding coordinate planes, and $\mathcal{C}_{\infty}$, which is placed at the infinite plane and can be regarded as the intersection of $\Hyp$ with the cone $-x^2+y^2+z^2=0$.

For each point $(x,y,z) \in \Hyp$, the equation \eqref{ConfFam1} has two solutions in $\lambda$ which we call the \emph{generalized Jacobi coordinates} or \emph{elliptic coordinates} of that point.
If $\mathcal{C}_0$ is a collared $\Hyp$-ellipse, those solutions are real and distinct, one belonging to $(-\infty,a]$, and the other to $[b,c]$, meaning that each point on $\Hyp$ is the intersection point of one collared $\Hyp$-ellipse and a confocal conic of hyperbolic type.
The generalized Jacobi coordinates of any point within $\mathcal{C}_0$ satisfy $0 < \lambda_1 \leq a$, $b \leq \lambda_2 \leq c$.

If $\mathcal{C}_0$ is a transverse $\Hyp$-ellipse, the equation will have two distinct real solutions within $12$ domains bounded by joint light-like tangent lines, one double real solution on those lines, and no real solutions within the remaining $8$ domains.
The generalized Jacobi coordinates of any point within $\mathcal{C}_0$ satisfy $b \leq \lambda_1 < 0 <\lambda_2 \leq a$.

\subsection{Billiards and periodic trajectories}
\label{BonH}

On $\Hyp$, we define the billiard motion as geodesic flow until the trajectory meets the boundary ellipse, then satisfying the billiard reflection law on it, using the bilinear form \eqref{eq:form} and the normal vector to the boundary at the point of reflection.
Note that the normal vector is not defined at those points of the boundary where the tangent line is light-like, so generally speaking, the reflection cannot be defined there.
On the other hand, as shown in \cite{KT}, when such a situation occurs on a conic, one can continuously extend the billiard flow to such points, defining the reflection there as returning along the same segment in the opposite direction.
For detailed discussion of the billiard reflection law in the pseudo-Euclidean setting, see \cites{KT,DR2012, DR2013}.

In \cites{V,MV} a method is proposed to determine the integrability of a discrete dynamical system by reducing the problem to the factorization of matrix polynomials. Specific applications include billiards in an ellipsoid in Euclidean and Minkowski spaces.
As shown in \cite{GaR}, the technique extends to $\Hyp$. 
The geometric manifestation of integrability can be seen
through the existence of caustics.
Namely, all segments of a given billiard trajectory within an $\Hyp$-ellipse are tangent to the same conic confocal with the boundary.
Since both the geodesic flow and the reflection preserve the type of the vector, we have that each billiard trajectory will be space-like, time-like, or light-like.
Moreover, that type will be the same for all trajectories sharing the same caustic, see Remarks \ref{rem:type} and \ref{rem:type2} of the present paper.

In general billiard problems, the study of periodic orbits and their geometric properties is of considerable interest. 
For integrable cases, the works \cites{GH, DR1998a, DR1998b, DR2011, DR2012, DR2019b, ADR2019} among many others characterize periodic trajectories in terms of an underlying elliptic curve and prove versions of a Poncelet-type theorem. 
Such a theorem will also hold in this paper's setting:
given a periodic billiard trajectory in an $\Hyp$-ellipse, any billiard trajectory which shares the same caustic is also periodic with the same period. 

The work of Cayley (see e.g. \cites{C1, C2}) in the $19^{th}$ century  on the Poncelet Theorem provided analytic conditions relating the period of a billiard trajectory to its caustics. A modern account of these results can be found the paper of Griffiths and Harris \cite{GH} from the 1970's.
Such conditions for a generalized Poncelet theorem within an ellipsoid in $d$-dimensional Euclidean spaces,
the Lobachevsky space, and pseudo-Euclidean spaces were derived in the last few decades, see \cites{DR1998a,DR1998b,DJR, DR2012}.

In this paper, we will use recently obtained Cayley-type conditions for elliptic billiards on $\Hyp$, which were derived using a divisor shift on the elliptic curve
\begin{equation}
Y^2 = \ep(X-a)(X-b)(X-c)(X-\nu) 
\label{eq:ellipticcurve}
\end{equation}
where $\ep = \sign(b\nu)$. 

\begin{theorem}[\cite{GaR}]\label{th:cayley}
A space-like or time-like billiard trajectory in the $\Hyp$-ellipse \eqref{eq:coneA} with the caustic $\mathcal{C}_{\nu}$ from the family \eqref{ConfFam1} is $n$-periodic if and only if
	\begin{equation*}
	\det\fourbyfour{B_{3}}{B_{4}}{B_{m+1}}{B_{4}}{B_{5}}{B_{m+2}}{B_{m+1}}{B_{m+2}}{B_{2m-1}}=0
\quad\text{and}\quad
n=2m \geq 4,
	\end{equation*}
or
	\begin{equation*}
	\det \fourbyfour{D_{2}}{D_{3}}{D_{m+1}}{D_{3}}{D_{4}}{D_{m+2}}{D_{m+1}}{D_{m+2}}{D_{2m}}=0
\quad\text{and}\quad
n=2m+1\ge3,
	\end{equation*}
where
	$$\sqrt{\ep(X-a)(X-b)(X-c)(X-\nu)} = B_0 + B_1X + B_2 X^2 + \cdots$$
	and 
	$$\sqrt{\frac{\ep(X-a)(X-b)(X-c)}{X-\nu}} = D_0 + D_1X + D_2 X^2 + \cdots$$
	are the Taylor expansions around $X=0$ and $\varepsilon=\sign(b\nu)$.
Furthermore, the only $2$-periodic trajectories are contained in the planes of symmetry. 
	
A light-like billiard trajectory in the  $\Hyp$-ellipse is $n$-periodic if and only if $n=2m \geq 4$ and
	\begin{equation*}
	\det\fourbyfour{E_{3}}{E_{4}}{E_{m+1}}{E_{4}}{E_{5}}{E_{m+2}}{E_{m+1}}{E_{m+2}}{E_{2m-1}}=0,
	\end{equation*}
	where
	$$\sqrt{\delta(X-a)(X-b)(X-c)} = E_0 + E_1X + E_2 X^2 + \cdots$$
	is the Taylor expansion around $X=0$ and $\delta = \sign(b)$. 
	
	\label{CayleyThm}
\end{theorem}

For more details, we refer to \cite{GaR}. Let us also observe a recent paper \cite{VW2020} devoted to geodesics on a hyperboloid in Euclidean space.

\subsection{Elliptic Periodic Trajectories}\label{sec:EllipticPeriodic}

The points in $\Mink$ that are symmetric about the coordinate planes have the same elliptic coordinates, hence there are eight points in $\Mink$ with elliptic coordinates $\lambda_1, \lambda_2$. 
Because of that symmetry, each billiard trajectory which is $n$-periodic in elliptic coordinates will also be periodic in the Cartesian coordinate system, but its period can be $n$ or $2n$. We define this scenario as follows. 

\begin{definition}
A billiard trajectory is $n$-\emph{elliptic periodic} if it is $n$-periodic in the elliptic coordinates corresponding to the confocal family (\ref{ConfFam1}). 
\end{definition}

Trajectories which directly connect two points with the same Jacobi elliptic coordinates will be 1-elliptic periodic. We thus consider $n$-elliptic periodicity for $n \geq 2$. Next, we derive algebro-geometric conditions for elliptic periodic billiard trajectories on $\Hyp$.  

\begin{theorem}
A billiard trajectory within the collared $\Hyp$-ellipse with caustic curve $\mathcal{C}_\nu$ is $n$-elliptic periodic and not $n$-periodic if and only if one of the following conditions is satisfied on the elliptic curve (\ref{eq:ellipticcurve}), where $Q_{\pm}$ are two points over $X=0$ and $P_\beta$ is a point over $X=\beta$: 
\begin{itemize}
\item $n=2m$ and 
	\begin{enumerate}[(i)]
	\item $\nu \in (-\infty,0) \cup (b,c) \cup (c,\infty)$ and $m(Q_- - Q_+) \sim 0$; 
	\item $\nu \in (-\infty, 0) \cup (c,\infty)$ and $(m+1)Q_- - (m-1)Q_+ - P_a - P_\nu \sim 0$; 
	\item $\nu \in (b, c)$ and $(m+1)Q_- - (m-1)Q_+ - P_a - P_c \sim 0$; 
	\end{enumerate} 
\item $n=2m+1$ and 
	\begin{enumerate}[(i)] \setcounter{enumi}{3}
	\item $\nu \in (-\infty,0) \cup (b,c) \cup (c,\infty)$ and $(m+1)Q_- - mQ_+ -P_a \sim 0$; 
	\item $\nu \in (-\infty, 0) \cup (c,\infty)$ and $(m+1)Q_- - mQ_+ - P_\nu \sim 0$; 
	\item $\nu \in (b, c)$ and $(m+1)Q_- - mQ_+ - P_c \sim 0$. 
	\end{enumerate}
\end{itemize}

A billiard trajectory within the transverse $\Hyp$-ellipse with caustic curve $\mathcal{C}_\nu$ is $n$-elliptic periodic and not $n$-periodic if and only if one of the following conditions is satisfied on the elliptic curve (\ref{eq:ellipticcurve}): 
\begin{itemize}
\item $n=2m$ and 
	\begin{enumerate}[(i)]\setcounter{enumi}{6}
	\item $\nu \in \R{}\setminus\{b,0,a,c\}$ 
	and $m(Q_- - Q_+) \sim 0$; 
	\item $\nu \in (-\infty, b) \cup (a, c) \cup (c,\infty)$ and $(m+1)Q_+ - (m-1)Q_- - P_a - P_b \sim 0$; 
	\item $\nu \in (b, 0)$ and $(m+1)Q_+ - (m-1)Q_- - P_a - P_\nu \sim 0$; 
	\item $\nu \in (0, a)$ and $(m+1)Q_+ - (m-1)Q_- - P_b - P_\nu \sim 0$; 
	\end{enumerate} 
\item $n=2m+1$ and 
	\begin{enumerate}[(i)] \setcounter{enumi}{10}
	\item $\nu \in (-\infty,b) \cup (b,0) \cup (a, c) \cup (c,\infty)$ and $(m+1)Q_+ - mQ_- -P_a \sim 0$; 
	\item $\nu \in (-\infty,b) \cup (0,a) \cup (a, c) \cup (c,\infty)$ and $(m+1)Q_+ - mQ_- -P_b \sim 0$. 
	\end{enumerate}
\end{itemize}
\end{theorem}

\begin{proof}
Let $\mathcal{P}(x) = \ep (x-a)(x-b)(x-c)(x-\nu)$ with $\ep = \sign(b\nu)$ and consider the differential equation 
\begin{equation}
\frac{d\lambda_1}{\sqrt{\mathcal{P}(\lambda_1)}} + \frac{d\lambda_2}{\sqrt{\mathcal{P}(\lambda_2)}} = 0
\label{eq:diffeq}
\end{equation}
along a given billiard trajectory. If $p_0$ is the initial point of an $n$-elliptic periodic trajectory and $p_1$ is the next point along the trajectory with the same elliptic coordinates as $p_0$, integrating the differential equation (\ref{eq:diffeq}) from $p_0$ to $p_1$ results in 
\begin{align*}
m_0(P_a-P_0) + m_1(P_c-P_b) &\sim 0, \text{ or } \\
m_2(P_\nu-P_0) + m_3(P_c-P_b) &\sim 0, \text{ or } \\
m_4(P_a-P_0) + m_5(P_\nu-P_b) &\sim 0 
\end{align*}
in the case of the collared $\Hyp$-ellipse. The period $n = m_0 = m_2=m_4$ and the $m_i$'s can be even or odd. Cases $(i)$ and $(v)$ follow from the proof of Theorem 5.5 in \cite{GaR}, where the $m_i$ are all even (case $(i)$), and  $m_0, m_1, m_2$ are odd and $m_3$ is even (case $(v)$). 

We prove $(ii)$, and note that the rest of the cases follow from a similar argument depending upon the parity of the $m_i$'s. Suppose $n=m_0$ is even and $m_1$ is odd. Then 
\begin{align*}
0 &\sim m_0 (P_a-P_0) + m_1(P_c-P_b) \\
   &\sim 2k_0(P_a - P_0) + 2k_1(P_c-P_b) + P_c-P_b \\
   &\sim k_0 (Q_- +Q_+) - k_0(2Q_+) + P_c - P_b \\
   &\sim k_0(Q_- -Q_+) + P_c - P_b \\
   &\sim k_0(Q_- -Q_+) + (Q_- + Q_+) - P_a - P_\nu \\
   &\sim (k_0+1)Q_- - (k_0-1)Q_+ - P_a - P_\nu,
\end{align*}
which is equivalent to $(ii)$. In particular, we used the fact that $P_a$, $P_b$, $P_c$, and $P_\nu$ are branching points on (\ref{eq:ellipticcurve}), i.e. $2P_a \sim 2P_b \sim 2P_c \sim 2P_\nu \sim Q_- + Q_+$. These cases are also further discussed in Section \ref{RotNum} in the context of rotation numbers.

In the case of the transverse $\Hyp$-ellipse, Integrating the differential equation (\ref{eq:diffeq}) from $p_0$ to $p_1$ results in 
\begin{align*}
m_6(P_0-P_b) + m_7(P_a-P_0) &\sim 0, \text{ or } \\
m_8(P_0-P_\nu) + m_9(P_0-P_a) &\sim 0, \text{ or } \\
m_{10}(P_0-P_b) + m_{11}(P_0-P_\nu) &\sim 0.
\end{align*}
The period is now $n = m_{2i} + m_{2i+1}$ for $i=2,3,4$. Just as before, case $(vii)$ follows from the proof of Theorem 5.5 of \cite{GaR}, and the remainder of the cases are proved similarly to $(ii)$ above. 
\end{proof}

\begin{remark}
As noted in the above proof, the divisor conditions $(i)$, $(v)$, and $(vii)$ are identical to the divisor conditions for periodic billiard trajectories given in Theorem 5.5 of \cite{GaR}. However, in the case of the collared and transverse $\Hyp$-ellipses, conditions $(i)$ and $(vii)$ produce $2m$-periodic and $m$-elliptic periodic trajectories. And in the case of the collared $\Hyp$-ellipse, condition $(v)$ produces trajectories which are $(4m+2)$-periodic and $(2m+1)$-elliptic periodic. See Remark 5.7 of \cite{GaR} for further details. 
\end{remark}

The above divisor conditions lead to explicit Cayley-type conditions for elliptic periodicity. 

\begin{theorem}
A billiard trajectory within the collared $\Hyp$-ellipse with caustic curve $\mathcal{C}_\nu$ is $n$-elliptic periodic and not $n$-periodic if and only if one of the following conditions is satisfied:
\begin{enumerate}[(a)]
\item $\nu \in (-\infty, 0) \cup (b, c) \cup (c,\infty)$ and 
	\begin{equation*}
	\det\fourbyfour{B_{3}}{B_{4}}{B_{n+1}}{B_{4}}{B_{5}}{B_{n+2}}{B_{n+1}}{B_{n+2}}{B_{2n-1}}=0 \qquad \text{ for } n \geq 2,
	\end{equation*}
where the entries $B_i$ are given above in Theorem \ref{th:cayley}.

\item $\nu \in (-\infty, 0) \cup (c,\infty)$ and 
	\begin{equation*}
	\det\fourbyfour{F_{1}}{F_{2}}{F_{m}}{F_{2}}{F_{3}}{F_{m+1}}{F_m}{F_{m+1}}{F_{2m-1}}=0 \qquad \text{ for } n=2m \geq 2,
	\end{equation*}
or
	\begin{equation*}
	\det \fourbyfour{D_{2}}{D_{3}}{D_{m+1}}{D_{3}}{D_{4}}{D_{m+2}}{D_{m+1}}{D_{m+2}}{D_{2m}}=0 \qquad \text{ for } n=2m+1 \geq 3
	\end{equation*}
where $\displaystyle\sqrt{\frac{\ep(X-b)(X-c)}{(X-a)(X-\nu)}} = F_0 + F_1X + F_2X^2 + \cdots$ is the Taylor expansion around $X=0$ and the entries $D_i$ are given in Theorem \ref{th:cayley}. 

\item $\nu \in (b, c)$ and 
	\begin{equation*}
	\det\fourbyfour{G_{1}}{G_{2}}{G_{m}}{G_{2}}{G_{3}}{G_{m+1}}{G_{m}}{G_{m+1}}{G_{2m-1}}=0 \qquad \text{ for } n=2m \geq 4,
	\end{equation*}
or
	\begin{equation*}
	\det \fourbyfour{H_{2}}{H_{3}}{H_{m+1}}{H_{3}}{H_{4}}{H_{m+2}}{H_{m+1}}{H_{m+2}}{H_{2m}}=0 \qquad \text{ for } n=2m+1 \geq 5
	\end{equation*}
where $\displaystyle\sqrt{\frac{\ep(X-b)(X-\nu)}{(X-a)(X-c)}} = G_0 + G_1X + G_2X^2 + \cdots$ and $\displaystyle\sqrt{\frac{\ep(X-a)(X-b)(X-\nu)}{X-c}} = H_0 + H_1X + H_2X^2 + \cdots$ are the Taylor expansions around $X=0$.

\item $\nu \in (-\infty, 0) \cup (b, c) \cup (c,\infty)$ and 
	\begin{equation*}
	\det \fourbyfour{I_{2}}{I_{3}}{I_{m+1}}{I_{3}}{I_{4}}{I_{m+2}}{I_{m+1}}{I_{m+2}}{I_{2m}}=0 \qquad \text{ for } n=2m+1 \geq 3
	\end{equation*}
where $\displaystyle\sqrt{\frac{\ep(X-b)(X-c)(X-\nu)}{X-a}} = I_0 + I_1X + I_2X^2 + \cdots$ is the Taylor expansion around $X=0$. 
\end{enumerate}

A billiard trajectory within the transverse $\Hyp$-ellipse with caustic curve $\mathcal{C}_\nu$ is $n$-elliptic periodic and not $n$-periodic if and only if one of the following conditions is satisfied:
\begin{enumerate}[(a)]\setcounter{enumi}{4}
\item $\nu \in (-\infty, b) \cup (b, 0) \cup (0,a) \cup (a, c) \cup (c,\infty)$ and 
	\begin{equation*}
	\det\fourbyfour{B_{3}}{B_{4}}{B_{n+1}}{B_{4}}{B_{5}}{B_{n+2}}{B_{n+1}}{B_{n+2}}{B_{2n-1}}=0 \qquad \text{ for } n \geq 2,
	\end{equation*}
where the entries $B_i$ are given above in Theorem \ref{th:cayley}.

\item $\nu \in (-\infty,b) \cup (a, c) \cup (c,\infty)$ and 
	\begin{equation*}
	\det\fourbyfour{J_{1}}{J_{2}}{J_{m}}{J_{2}}{J_{3}}{J_{m+1}}{J_{m}}{J_{m+1}}{J_{2m-1}}=0 \qquad \text{ for } n=2m \geq 2,
	\end{equation*}
where $\displaystyle\sqrt{\frac{\ep(X-c)(X-\nu)}{(X-a)(X-b)}} = J_0 + J_1X + J_2X^2 + \cdots$ is the Taylor expansion around $X=0$.

\item $\nu \in (b,0)$ and 
	\begin{equation*}
	\det\fourbyfour{F_{1}}{F_{2}}{F_{m}}{F_{2}}{F_{3}}{F_{m+1}}{F_{m}}{F_{m+1}}{F_{2m-1}}=0 \qquad \text{ for } n=2m \geq 2,
	\end{equation*}
where the entries $F_i$ are given above in part (b).

\item $\nu \in (0,a)$ and 
	\begin{equation*}
	\det \fourbyfour{K_{1}}{K_{2}}{K_{m}}{K_{2}}{K_{3}}{K_{m+1}}{K_{m}}{K_{m+1}}{K_{2m-1}}=0 \qquad \text{ for } n=2m \geq 4
	\end{equation*}
where $\displaystyle\sqrt{\frac{\ep(X-a)(X-c)}{(X-b)(X-\nu)}} = K_0 + K_1X + K_2X^2 + \cdots$ is the Taylor expansion around $X=0$.

\item $\nu \in (-\infty, b) \cup (b, 0) \cup (a, c) \cup (c,\infty)$ and 
	\begin{equation*}
	\det \fourbyfour{I_{2}}{I_{3}}{I_{m+1}}{I_{3}}{I_{4}}{I_{m+2}}{I_{m+1}}{I_{m+2}}{I_{2m}}=0 \qquad \text{ for } n=2m +1 \geq 3
	\end{equation*}
where the entries $I_i$ are given above in part (d).

\item $\nu \in (-\infty, b) \cup (0,a) \cup (a, c) \cup (c,\infty)$ and 
	\begin{equation*}
	\det \fourbyfour{L_{2}}{L_{3}}{L_{m+1}}{L_{3}}{L_{4}}{L_{m+2}}{L_{m+1}}{L_{m+2}}{L_{2m}}=0 \qquad \text{ for } n=2m +1 \geq 3
	\end{equation*}
where $\displaystyle\sqrt{\frac{\ep(X-a)(X-c)(X-\nu)}{X-b}} = L_0 + L_1X + L_2X^2 + \cdots$ is the Taylor expansion around $X=0$.
\end{enumerate}
\end{theorem}

\begin{proof}
We prove part $(c)$ and note the proofs for the remaining parts are similar. Part $(c)$ uses divisor conditions $(iii)$ and $(vi)$ from the previous theorem. 

Consider first the case $n=2m$ and divisor condition $(iii)$: $(m+1)Q_- - (m-1)Q_+ - P_a - P_c \sim 0.$ This divisor condition is equivalent to the existence of a meromorphic function with a zero of order $m+1$ at $Q_-$, a pole of order $m-1$ at $Q_+$, and simple poles at $P_a$ and $P_b$. A basis of $\mathcal{L}((m-1)Q_{+} + P_a + P_c)$ is $\{1,f_1, \ldots, f_m\}$ where $$f_k = \frac{y-G_0 - G_1 x - \cdots - G_k x^k}{x^{k-1}}.$$ The existence of such a function is equivalent to the given determinant condition. 

Now consider $n=2m+1$ and divisor condition $(vi)$: $(m+1)Q_- - mQ_+ - P_c \sim 0$. This divisor condition is equivalent to the existence of a meromorphic function with a zero of order $m+1$ at $Q_-$, a pole of order $m$ at $Q_+$, and a simple pole at $P_c$. A basis of the space of such functions $\mathcal{L}(mQ_{+} + P_c)$ is $\{1,g_1, \ldots, g_m\}$ where $$g_k = \frac{y-H_0 - H_1 x - \cdots - H_k x^k}{x^{k}}.$$ The existence of such a function is equivalent to the second determinant condition in part $(c)$. 
\end{proof}

\begin{example}[2-Elliptic Periodic]
In the collared $\Hyp$-ellipse, a 2-elliptic periodic trajectory can be found by satisfying either $B_3=0$ or $F_1=0$, which are equivalent to 
$$ (a b c+(a b  -a c  -b c) \nu ) (a b c + (-a b + ac -b c) \nu ) (a b c + (-a b -a c +b c) \nu )=0$$ 
and 
$$a b c + (-a b -a c + bc) \nu =0,$$ 
respectively. Clearly any solution to $F_1=0$ will also satisfy $B_3=0$. 

In the transverse $\Hyp$-ellipse, a 2-elliptic periodic trajectory can be found by satisfying either $B_3=0$, $F_1=0$, $J_1=0$. The first two are given above, and the third is equivalent to 
$$ a b c+(a b -a c -b c) \nu=0.$$
Any solution to $F_1=0$ or $J_1=0$ will also be a solution to $B_3=0$. 

Pictures of such trajectories are shown in Figure \ref{fig:2EP}.
\end{example}

\begin{figure}[thp]
\centering
\begin{tabular}{c c}
a) \includegraphics[width=0.40\textwidth]{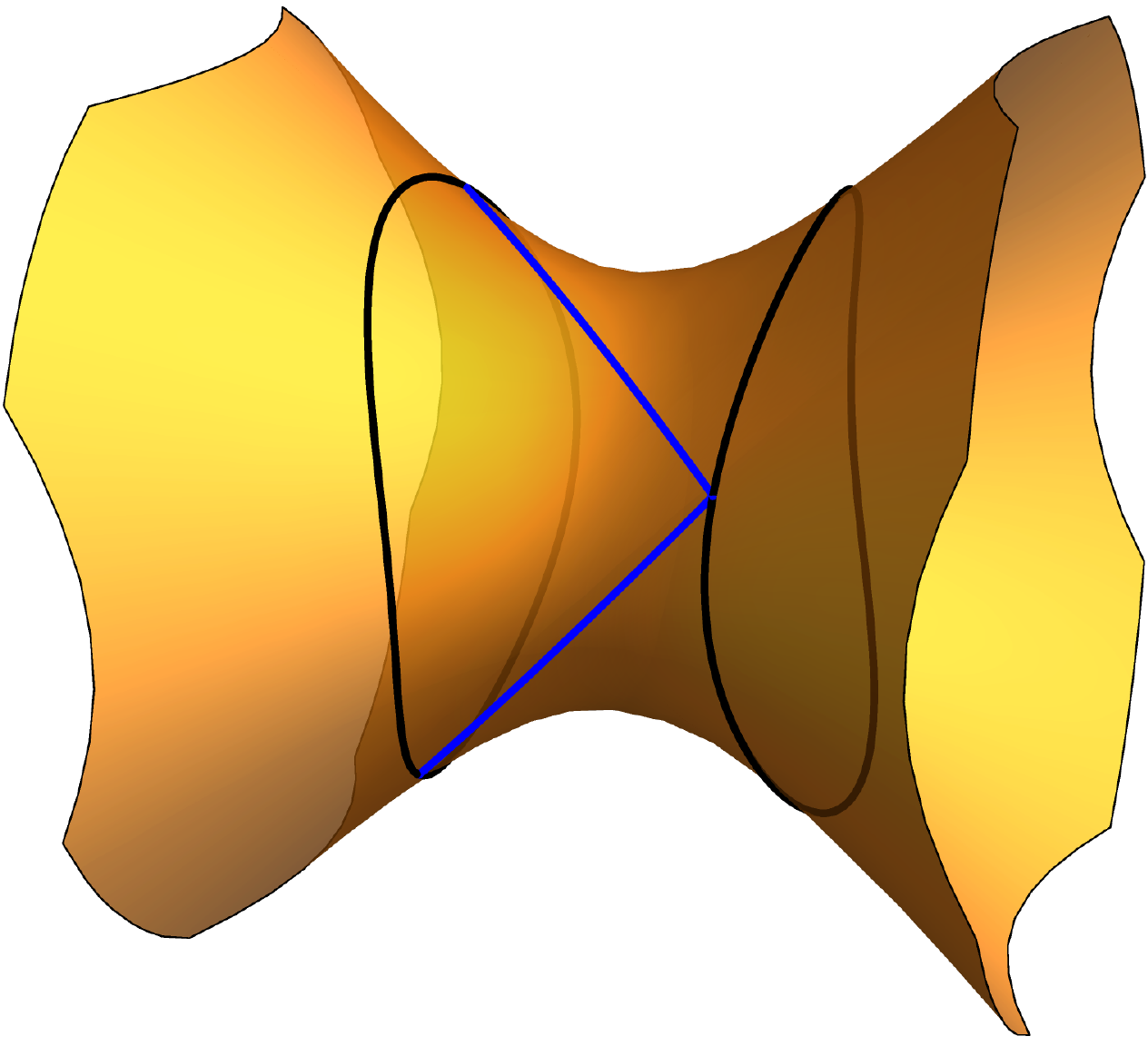} & b) \includegraphics[width=0.40\textwidth]{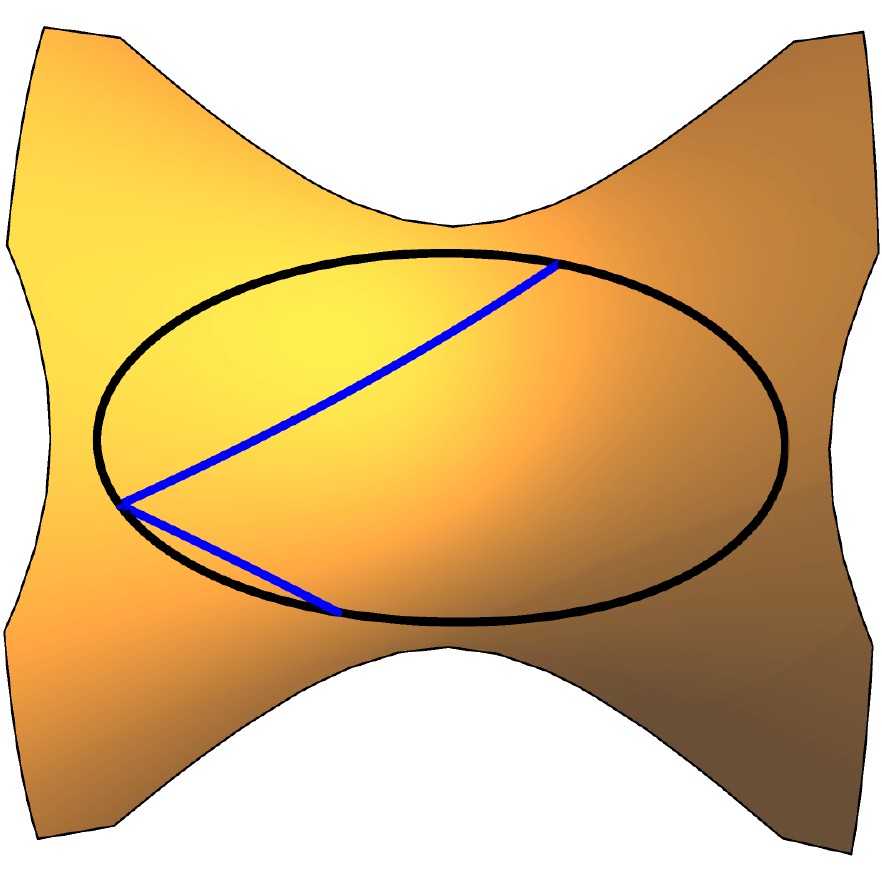}
\end{tabular}
\caption{Two possible trajectories that are 2-elliptic periodic and not 2-periodic in the collared (a) and transverse (b) $\Hyp$-ellipses. }
\label{fig:2EP}
\end{figure}

\begin{example}[3-Elliptic Periodic]
In the collared $\Hyp$-ellipse, a 3-elliptic periodic trajectory can be found by satisfying either
$$\det\twobytwo{B_3}{B_4}{B_4}{B_5}=0, \text{ or } D_2=0, \text{ or } I_2=0.$$
This is equivalent to 
\begin{align*}
0&=\left[\left(-3 a^2 b^2+c^2 (a-b)^2+2 a b c (a+b)\right)\nu ^2 +2 a b c (a b-ac-bc) \nu+ (abc)^2\right] \\
&\times \left[(-a^2 (b-c)^2+2 a b c (b+c)-b^2 c^2)\nu^2 - 2abc(ab + ac + bc)\nu + 3(abc)^2\right] \\
&\times \left[(a^2 (b-c)^2+2 a b c (b+c)-3 b^2 c^2)\nu^2 + 2abc (-ab -ac +bc )\nu + (abc)^2\right] \\
&\times \left[(a^2 (b-c) (b+3 c)+2 a b c (c-b)+b^2 c^2)\nu^2 + 2abc(-ab + ac - bc)\nu + (abc)^2 \right] 
\end{align*}
or 
$$0=3(abc)^2\ -2 abc\left(ab + bc + ac \right)\nu + \left( 4 a b c (a+b+c)-(a b+a c+b c)^2\right) \nu^2, $$
or 
$$ 0= (abc)^2 -2abc(-bc+ab+ac)\nu + (a^2 (b-c)^2+2 a b c (b+c)-3 b^2 c^2) \nu^2,$$
respectively.

In the transverse $\Hyp$-ellipse, a 3-periodic trajectory can be found by satisfying either 
$$\det\twobytwo{B_3}{B_4}{B_4}{B_5}=0, \text{ or } I_2=0, \text{ or } L_2=0.$$
The first two conditions are given above while the third is equivalent to 
$$0 = (abc)^2 -2 a b c (a b-a c+b c)\nu + (b^2 c^2+2 a b c (c-b)+a^2(b^2+2 b c-3 c^2))\nu^2 $$
Pictures of such 3-elliptic periodic trajectories are shown in Figure \ref{fig:3EP}.
\end{example}

\begin{figure}[thp]
\centering
\begin{tabular}{c c}
a) \includegraphics[width=0.40\textwidth]{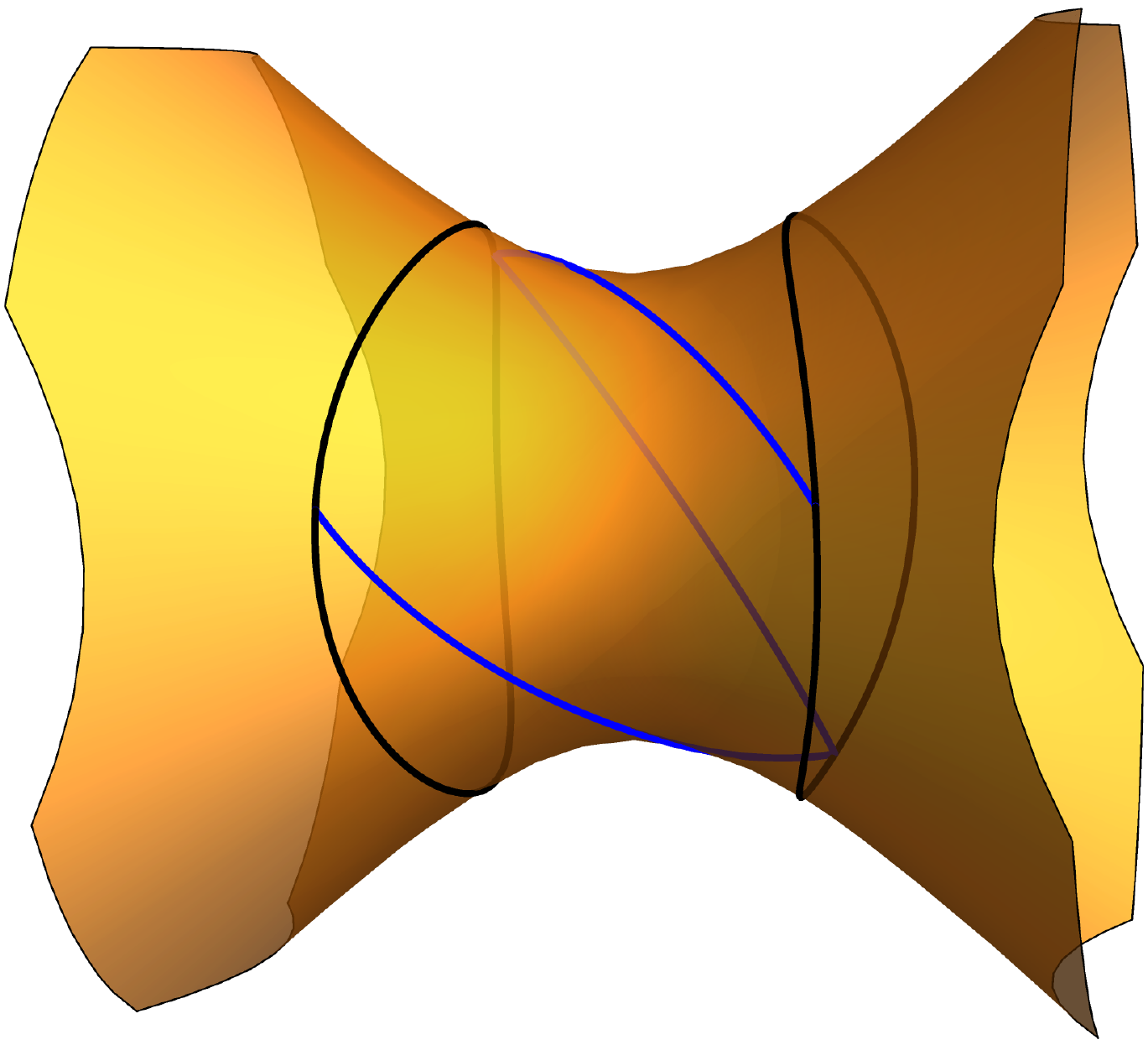} & b) \includegraphics[width=0.40\textwidth]{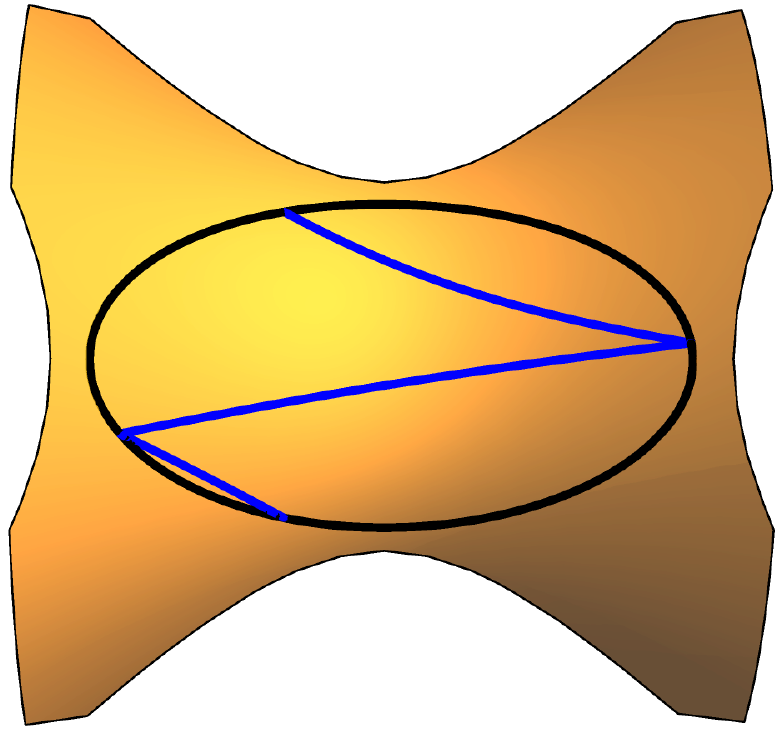}
\end{tabular}
\caption{Two possible trajectories that are 3-elliptic periodic and not 3-periodic in the collared (a) and transverse (b) $\Hyp$-ellipses. }
\label{fig:3EP}
\end{figure}

\newpage

\section{Topological properties of billiards in confocal families}
\label{TopProperties}

As described in the previous section, billiards  within confocal conics on $\Hyp$ come in two distinct geometric types. In this section we present a topological description of billiards in each setting using Fomenko invariants (see \cites{BMF1990, BF2004, BBM2010}).
Such descriptions have been made for elliptic billiards in the Euclidean plane \cites{DR2009, DR2010}, domains bounded by confocal parabolas \cite{Fokicheva2014}, with Hooke's potential \cite{R2015}, the Minkowski plane and geodesics on ellipsoids in $\Mink$ \cite{DR2017}, non-convex billiards \cite{Ved2019}, billiards with slipping \cite{FVZ2021}, and broader classes of billiards and Hamiltonian impact systems \cites{VK2018, FV2019, PRK2020}.
For the larger body of work on the topic, see also the references therein.

\subsection{Transverse $\Hyp$-ellipse} In the case of the transverse $\Hyp$-ellipse, the constants satisfy $b < 0 < a < c$, while the confocal curves are of elliptic-type if $\lambda \in (b,a)\cup(a,c)$ and hyperbolic-type if $\lambda \in (-\infty, b] \cup [c,\infty)$, see Section \ref{sec:confocal} for details.

The billiard table $\mathcal{T}$ will be the set of all points on and in the interior of the transverse $\Hyp$-ellipse $\mathcal{C}_0$ with $z>0$.
Topologically, $\mathcal{T}$ is homeomorphic to the closed planar disk. Consider a point $P \in \mathcal{C}_0$ and suppose $u,v \in T_P\Hyp$ are the unit vectors, which is itself homeomorphic to $\Sn{1}$. 
Let $\sim$ be the equivalence relation on the solid torus $\mathcal{T} \times \Sn{1}$ defined by 
$$ (P,u) \sim (P,v)  \iff  P \in \mathcal{C}_0 \text{ and } u,v \in \Sn{1} \text{ reflect to one another off } \mathcal{C}_0.$$ 

Every billiard trajectory inside $\mathcal{T}$ induces a trajectory in $\mathcal{T} \times \Sn{1} / \sim$. Further, this correspondence in trajectories induces a projection of the billiard phase space into $\mathcal{T} \times \Sn{1} / \sim$ that preserves the trajectories and leaves of the Liouville foliation. 

\begin{theorem}
	The manifold $\mathcal{T} \times \Sn{1}/\sim$ is represented by the Fomenko graph in Figure \ref{TransverseFomenko}.
\end{theorem}

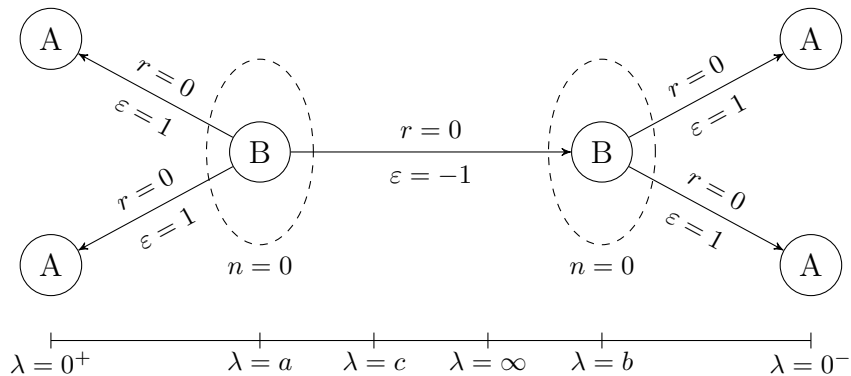
\begin{figure}[th]
	\centering
	\begin{tikzpicture}[>=stealth',el/.style = {inner sep=5pt, align=left, sloped}]
	\tikzset{vertex/.style = {shape=circle,draw,minimum size=1.5em}}
	\node[vertex] (aa) at  (0,1.5) {A};
	\node[vertex] (bb) at  (0,-1.5) {A};
	\node[vertex] (cc) at  (2.75,0) {B};
	\node[vertex] (dd) at  (7.25,0) {B};
	\node[vertex] (ee) at  (10,1.5) {A};
	\node[vertex] (ff) at  (10,-1.5) {A};
	\path[->]
	(cc) edge node[el,below]  {\footnotesize $\varepsilon=1$} node[el, above] {\footnotesize $r=0$} (aa)
	(cc) edge node[el,below]  {\footnotesize $\varepsilon=1$} node[el, above] {\footnotesize $r=0$} (bb)
	(cc) edge node[el,below]  {\footnotesize $\varepsilon=-1$} node[el, above] {\footnotesize $r=0$} (dd)
	(dd) edge node[el,below]  {\footnotesize $\varepsilon=1$} node[el, above] {\footnotesize $r=0$} (ee)
	(dd) edge node[el,below]  {\footnotesize $\varepsilon=1$} node[el, above] {\footnotesize $r=0$} (ff);
	\draw (0,-2.5) -- (10,-2.5);
	\foreach \x in {0,2.75,4.25,5.75,7.25,10}
	\draw[shift={(\x,-2.5)},color=black] (0pt,3pt) -- (0pt,-3pt);
	\node[below] at (0,-2.5) {\footnotesize $\lambda=0^+$};
	\node[below] at (2.75,-2.5) {\footnotesize $\lambda=a$};
	\node[below] at (4.25,-2.5) {\footnotesize $\lambda=c$};
	\node[below] at (5.75,-2.5) {\footnotesize $\lambda=\infty$};
	\node[below] at (7.25,-2.5) {\footnotesize $\lambda=b$};
	\node[below] at (10,-2.5) {\footnotesize $\lambda=0^-$};
	\draw[dashed] (2.75,0) ellipse (20pt and 35pt);
	\node at (2.75,-1.5) {\footnotesize $n=0$};
	\draw[dashed] (7.25,0) ellipse (20pt and 35pt);
	\node at (7.25,-1.5) {\footnotesize $n=0$};
	
	\end{tikzpicture}
	\caption{Fomenko graph for the billiard in the transverse $\Hyp$-ellipse with $z>0$. }
	\label{TransverseFomenko}
\end{figure}

\begin{proof}
Each level set of the manifold corresponds to billiard motion with a fixed confocal curve as a caustic, $\mathcal{C}_{\lambda}$ for $\lambda \in \R{} \cup \{\infty\}$.
	
When $\lambda \notin \{a,b,c\}$, the level sets are nondegenerate. The level set is a single torus when the caustic is hyperbolic-type -- that is, when $\lambda \in (-\infty, b) \cup (c,\infty) \cup \{\infty\}$. The level set is a union of two tori when the caustic is elliptic-type and does intersects $\mathcal{C}_0$ at four distinct points (i.e. when $\lambda \in (b,a)$). And the level set is again a single torus when $\mathcal{C}_{\lambda}$ is elliptic-type and does not intersect $\mathcal{C}_0$ (i.e. when $\lambda \in (a,c)$). 
	
If $\lambda = a$, then the level set contains a single closed trajectory, which is two-periodic and contained in the plane $x=0$, and two homoclinic separatrices.
A trajectory on any of the separatrices is placed on one side of the coordinate $yz$-plane and its segments alternately contain foci $F_{-+}^x$ and $F_{++}^x$.
	
As $\lambda \to a^-$, the caustic is elliptic-type and intersects $\mathcal{C}_0$. Trajectories can be in one of two regions bounded by $\mathcal{C}_0$ and $\mathcal{C}_{\lambda}$, so the level set is the union of two tori. As $\lambda \to a^+$, the caustic is elliptic-type but does not intersect $\mathcal{C}_0$, and hence the level set is a single torus. This collection of level sets is represented by the Fomenko atom B. A similar analysis at $\lambda =b$ provides an analogous result. 
	
If $\lambda = c$, the segments of a trajectory have a simple geometric description. 
The plane which determines each billiard segment alternately contains antipodal pairs of foci (i.e. the plane determining one segment will contain $F_{++}^z$ and $F_{--}^z$ while the plane determining the next segment of the trajectory will contain $F_{+-}^z$ and $F_{-+}^z$). 
Such trajectories do not limit to periodic trajectories as before. While $\lambda=c$ is a transition point for the caustic to change from hyperbolic- to elliptic-type, neither caustic curves intersect $\mathcal{C}_0$, and the billiard itself does not fundamentally change at $\lambda=c$.
	
If $\lambda = \infty$, the trajectories are all light-like, i.e.~their segments are placed along generatrices of $\Hyp$. Their behavior is qualitatively identical to when $\lambda \in (-\infty, b) \cup (a, \infty)$.
	
In the neighborhood of $\lambda=b$, the analysis is similar to the case $\lambda=a$.

Consider the limiting case $\lambda =0$. As $\lambda \to 0^-$, the billiard motion is in one of two regions bounded by $\mathcal{C}_0$ and $\mathcal{C}_{\lambda}$, each lying on one side of the plane $y=0$. The limiting motion along the boundary is periodic: the trajectory moves along the time-like arc of the boundary. This periodic motion is represented by the two A atoms in Figure \ref{TransverseFomenko}. As $\lambda \to 0^+$, the same analysis is true except that the two regions between $\mathcal{C}_0$ and $\mathcal{C}_{\lambda}$ are on opposite sides of the plane $x=0$, and the limiting periodic trajectories are space-like arcs of $\mathcal{C}_0$. 
\end{proof}

\begin{remark}\label{rem:type}
We note that the trajectories with the caustics $\mathcal{C}_{\lambda}$ such that $\lambda>0$ are space-like, while they are time-like for $\lambda<0$.
\end{remark}

\subsection{Collared $\Hyp$-ellipse} In the case of the collared $\Hyp$-ellipse, the constants satisfy $0 < a < b < c$ and the confocal curves are of elliptic-type if $\lambda \in (-\infty,a)$ and hyperbolic-type if $\lambda \in (b,c)$, see Section \ref{sec:confocal} for details.

Let $\mathcal{E}$ be the billiard table; that is, all points on and in the interior of the collared $\Hyp$-ellipse $\mathcal{C}_0$. Topologically, $\mathcal{E}$ is homeomorphic to the closed annulus. Consider a point $P \in \mathcal{C}_0$ and suppose $u,v \in T_P\Hyp$, where $u$, $v$ are unit vectores, which is itself homeomorphic to $\Sn{1}$. 
Let $\sim$ be the equivalence relation on hollow torus with thickened walls $\mathcal{E} \times \Sn{1}$ defined by 
$$ (P,u) \sim (P,v)  \iff  P \in \mathcal{C}_0 \text{ and } u,v \in \Sn{1} \text{ reflect to one another off } \mathcal{C}_0.$$ 

Every billiard trajectory inside $\mathcal{E}$ induces a trajectory in $\mathcal{E} \times \Sn{1} / \sim$. Further, this correspondence in trajectories induces a projection of the billiard phase space into $\mathcal{E} \times \Sn{1} / \sim$ that preserves the trajectories and leaves of the Liouville foliation. 

\begin{theorem}
	The manifold $\mathcal{E} \times \Sn{1} /\sim$ is represented by the Fomenko graph in Figure \ref{CollaredFomenko}.
\end{theorem}

\begin{figure}[th]
	\centering
	\begin{tikzpicture}[>=stealth',el/.style = {inner sep=3pt, align=left, sloped},em/.style = {inner sep=3pt, pos=0.75, sloped}]
	\tikzset{vertex/.style = {shape=circle,draw,minimum size=1.5em}}
	\node[vertex] (aa) at  (0,1.5) {A};
	\node[vertex] (bb) at  (5,0) {$\text{C}_2$};
	\node[vertex] (ee) at  (10,1.5) {A};
	\node[vertex] (ff) at  (0,-1.5) {A};
	\node[vertex] (jj) at  (10,-1.5) {A};
	\path[->] 
	(bb) edge node[el,below]  {\footnotesize $\varepsilon=1$} node[el, above] {\footnotesize $r=0$} (aa) 
	(bb) edge node[el,below]  {\footnotesize $\varepsilon=1$} node[el, above] {\footnotesize $r=\infty$} (ee) 
	(bb) edge node[el,below]  {\footnotesize $\varepsilon=1$} node[el, above] {\footnotesize $r=0$} (ff) 
	(bb) edge node[el,below]  {\footnotesize $\varepsilon=1$} node[el, above] {\footnotesize $r=\infty$} (jj); 
	\draw (0,-3) -- (10,-3);
	\foreach \x in {0,5,6.67,8.33,10}
	\draw[shift={(\x,-3)},color=black] (0pt,3pt) -- (0pt,-3pt);
	\node[below] at (0,-3) {\footnotesize $\lambda=b$};
	\node[below] at (5,-3) {\footnotesize $\lambda=c$};
	\node[below] at (6.67,-3) {\footnotesize $\lambda=\infty$};
	\node[below] at (8.33,-3) {\footnotesize $\lambda=0$};
	\node[below] at (10,-3) {\footnotesize $\lambda=a$};
	\end{tikzpicture}
	\caption{Fomenko graph for the billiard in the collared $\Hyp$-ellipse. }
	\label{CollaredFomenko}
\end{figure}
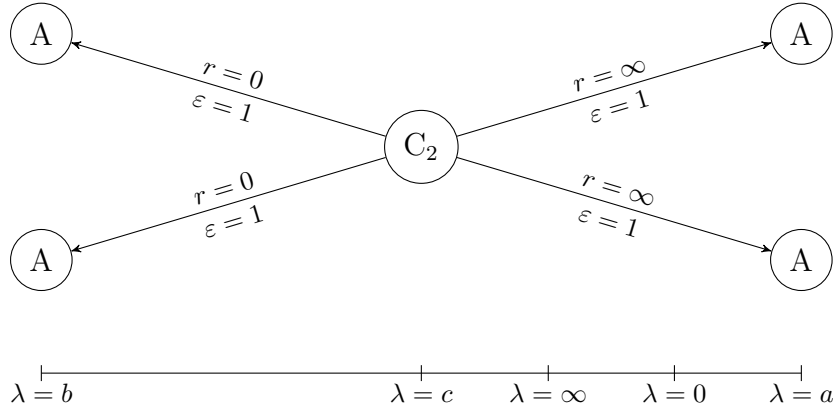

\begin{proof}
	Just as before, each level set of the manifold corresponds to billiard motion with a fixed confocal curve as a caustic, $\mathcal{C}_{\lambda}$ for $\lambda \in (-\infty, a] \cup [b,\infty) \cup \{\infty\}$. 
	
	When $\lambda \notin \{a,b,c\}$, the level sets are nondegenerate. The level set is a union of two tori when the caustic is hyperbolic-type -- that is, when $\lambda \in (b, c)$. There, each torus corresponds to a billiard trajectory in one of two regions, symmetric about the plane $z=0$, bounded by the confocal curves $\mathcal{C}_{\lambda}$ and the boundary $\mathcal{C}_0$. 

In the case $\lambda \in (c,+\infty)$, there are no confocal curves on $\Hyp$, as the intersection of the cone (\ref{ConfFam1}) with the hyperboloid is empty. However, since the corresponding geodesics are intersections of the planes tangent to the cone with $\Hyp$, we can still determine the corresponding billiard segments. In this case, the level set is the union of two tori, one for each direction a billiard trajectory can wind around $\mathcal{E}$.

At $\lambda =\infty$, the trajectories are light-like and each torus corresponds to a trajectory that winds around $\mathcal{E}$ in either the overhand or underhand direction. 
	
In the cases $\lambda \in (-\infty,0) \cup (0,a)$, the level set remains the union of two tori, one for each direction a billiard can wind around $\mathcal{E}$. Each has a different geometric interpretation. If $\lambda \in (-\infty,0)$, the caustic is elliptic-type and outside $\mathcal{E}$, while $\lambda \in (0,a)$ corresponds to motion with an elliptic-type caustic inside $\mathcal{E}$. This motion is remarkable in the following sense: every torus corresponding to $\lambda \in (0,a)$ is simply periodic geodesic flow of a space-like trajectory. The billiard motion does not reach the boundary, thus each trajectory there is closed, see Figure \ref{PeriodicResonant}. Thus, every torus in such level set is resonant.  In the limiting case $\lambda \to 0^\pm$, the level set is again the union of two resonant tori, again each corresponding to the direction a trajectory can wind around $\mathcal{E}$. 

\begin{figure}[ht]
	\centering
	\includegraphics[width=6.4cm,height=6.3cm]{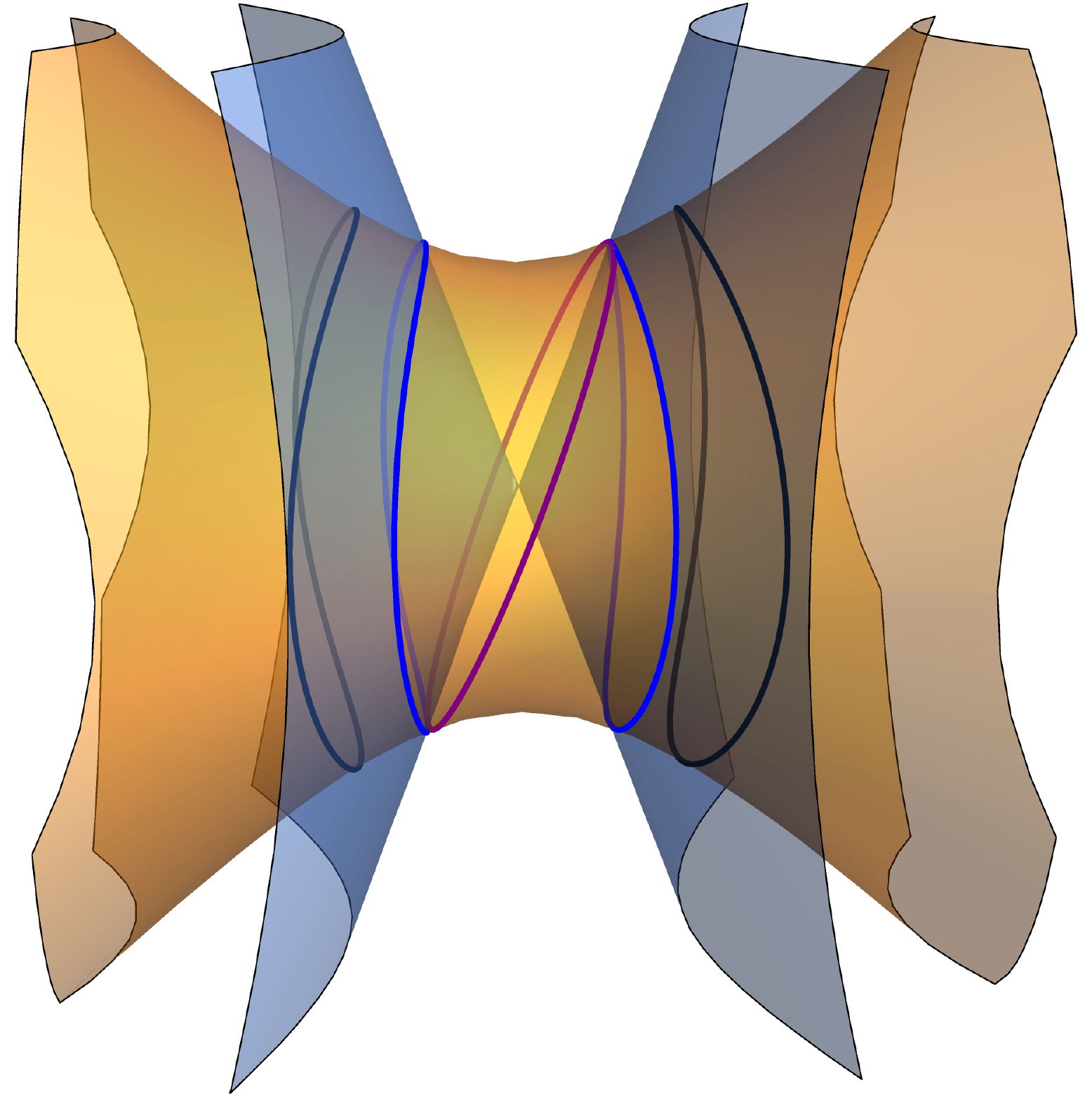}
	\caption{A periodic billiard trajectory tangent to a confocal curve without reflection from the boundary, $\mathcal{C}_0$. In the language of section \ref{sec:EllipticPeriodic}, this is a 1-elliptic periodic trajectory}
	\label{PeriodicResonant}
\end{figure}

Now consider the degenerate cases. If $\lambda =b$, the billiard is periodic and contained in the plane $y=0$, alternately reflecting off of each boundary component of $\mathcal{E}$. There will be two A atoms for this limiting motion, one for each of the two periodic trajectories. As $\lambda \to b^+$, the level set is the union of two tori, as described above for $\lambda \in (b,c)$.
	
If $\lambda = c$, the billiard is periodic and contained in the plane $z=0$ and the level set is degenerate. As $\lambda \to c^-$ the level set is the union of two tori, one for each region of the table where trajectories are tangent to a caustic curve of hyperbolic-type. As $\lambda \to c^+$, the level set is also a union of two tori, each corresponding to which direction the trajectory winds around $\mathcal{E}$. As there are four circles and four separatrices at this point, there will be a $\text{C}_2$ atom at $\lambda = c$.
	
If $\lambda =a$, the billiard is periodic, space-like, and wraps around the collar of $\Hyp$ in the plane $x=0$ and the level set is degenerate. Each winding direction of this periodic motion corresponds to each A atom. As $\lambda \to a^-$ the motion is the same as discussed above for $\lambda \in (0,a)$ and is the union of two tori. The level set is empty as $\lambda \to a^+$. 

Recall the notation of the gluing matrix and basis cycles on tori, $\ilvector{\lambda^+}{\mu^+} = \iltwobytwo{\alpha}{\beta}{\gamma}{\delta}\ilvector{\lambda^-}{\mu^-}$ \cites{BMF1990,BF2004}.  The gluing matrix for the two left edges is $A_L=\iltwobytwo{0}{1}{1}{0}$. The new basis cycles $(\lambda^+,\mu^+)$ correspond geometrically to the motion around the collar of $\mathcal{E}$ and the 2-periodic trajectory in the plane $y=0$, respectively. For the two right edges, the gluing matrix is $A_R=\iltwobytwo{1}{0}{0}{-1}$. Geometrically, the non-contractible basis cycle $\lambda^+$ corresponds to the circular periodic trajectory around the collar in the plane $x=0$ and the other basis cycle $\mu^+$ is complementary to $\lambda^+$.
\end{proof}

\begin{remark}\label{rem:type2}
In this case, the space-like trajectories correspond to the caustics $\mathcal{C}_{\lambda}$ with $\lambda < a$, while time-like to $\lambda>b$.
Compare with Remark \ref{rem:type}.
\end{remark}

\subsection{Remarks on equivalent systems}

We note that the Fomenko graph in Figure \ref{TransverseFomenko} for the case of billiards in the transverse $\Hyp$-elllipse is identical to the Fomenko graph corresponding to billiards inside an ellipse in the Minkowski plane \cite{DR2017}. Therefore, because their marked Fomenko graphs are identical, these two billiard systems are Liouville equivalent. 

In the case of the collared $\Hyp$-ellipse, the marked molecule is the same as a particular instance of the Chaplygin case of integrable rigid-body dynamics (see \cite{FN2015}). Therefore these integrable systems are Liouville equivalent.

\subsection{Billiard books}

A set of conjectures by Fomenko outlines a connection between integrable systems and elliptical billiards in the Euclidean plane. By considering domains that are subsets of the ellipse which are bounded by arcs of confocal hyperbolas and ellipses, it is conjectured that any integrable system can be represented by gluing copies of such domains together in a suitable fashion, and obtaining a generalized billiard domain called a \emph{billiard book}. See \cite{VK2018} for details. 

In the case of the transverse $\Hyp$-ellipse, the unmarked version of the Fomenko graph in Figure \ref{TransverseFomenko} is represented by a billiard book which is constructed by gluing together two copies of the billiard book representing the 3-atom B. See \S 2 of \cite{FKK2020} for a direct construction of this billiard book. 

In the case of the collared $\Hyp$-ellipse, the unmarked version of the Fomenko graph in Figure \ref{CollaredFomenko} is represented by a billiard book representing the 3-atom $\text{C}_2$.  One can follow Algorithm 1 of \cite{VK2018} to create such a billiard book, and the necessary details, using the notation from that work, are as follows. This book is made from four sheets, $A_0^\prime$. The permutations $\sigma_3 = \sigma_4 = \text{id}$, while $\sigma_1 =  (12)(34)$ and $\sigma_2 = (14)(23)$ define the gluing of the edges.

\begin{figure}[th]
	\centering
	\begin{tikzpicture}[>=stealth'] 
	
	\node [above] at (-1,1.5) {$(1)$};
	\node [below] at (0.5,0) {$(12)$};
	\node [above] at (2,1.5) {$(234)$};
	\node [below] at (2,-1.5) {$(34)$};
	\node [left] at (-0.25,0.75) {\footnotesize{1}};
	\node [left] at (1.20,0.75) {\footnotesize{2}};
	\node [left] at (1.85,-1.15) {\footnotesize{3}};
	\node [right] at (2.15,-1.15) {\footnotesize{4}};
	
	\draw[thick] (-1,1.5) -- (0.5,0);
	\draw[thick] (0.5,0) -- (2,1.5);
	\draw[thick] (2,1.5) to [bend left = 30] (2,-1.5);
	\draw[thick] (2,1.5) to [bend right = 30] (2,-1.5);
	
	\node [above] at (5,1.5) {$(14)$};
	\node [above] at (8,1.5) {$(23)$};
	\node [below] at (5,-1.5) {$(12)$};
	\node [below] at (8,-1.5) {$(34)$};
	\node [left] at (5,-0.75) {\footnotesize{1}};
	\node [left] at (6.25,-0.75) {\footnotesize{2}};
	\node [right] at (6.7,-0.75) {\footnotesize{3}};
	\node [right] at (8,-0.75) {\footnotesize{4}};
	
	\draw[thick] (5,1.5) -- (5,-1.5) -- (8,1.5) -- (8,-1.5) -- (5,1.5);
	
	\end{tikzpicture}
	\caption{A side view of the two billiard books corresponding to the Fomenko graphs of the transverse (left) and collared (right) $\Hyp$-ellipse. Edges are labeled with the numbering of the sheet while vertices are labeled with their permutations.}
	\label{BilliardBook}
\end{figure}
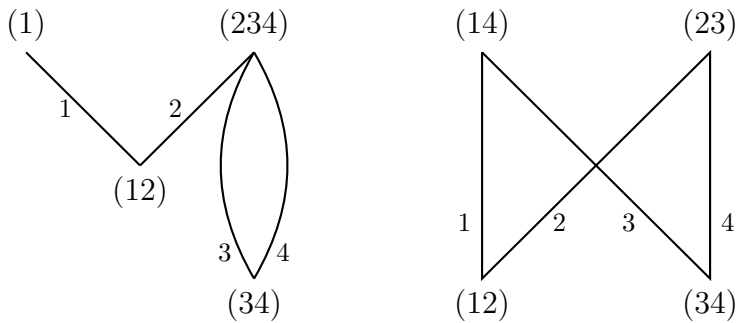

\section{Discriminantly factorizable and separable polynomials}
\label{sec:discriminantly-poly}

In the Euclidean plane \cite{DR2019a} and Minkowski plane \cite{ADR2019}, it is shown that the Cayley-type conditions contain a rich algebro-geometric structure related to discriminantly separable polynomials, which were introduced in \cite{Dr}.

\begin{definition}[\cite{Dr}]
	A polynomial $F(x_1, \ldots, x_n)$ is \emph{discriminantly separable} if there exist polynomials $f_1(x_1), \ldots, f_n(x_n)$ such that the discriminant $\mathcal{D}_{x_i}F$ of $F$ with respect to $x_i$ satisfies $$\mathcal{D}_{x_i}F(x_1, \ldots, \widehat{x_i}, \ldots, x_n) = \displaystyle\prod_{j \neq i} f_j(x_j)$$ for each $1 \leq i \leq n$. 
\end{definition}
 Now we examine the periodicity conditions of Theorem \ref{th:cayley} from that perspective, and note that the elliptic-periodicity conditions can be treated similarly. The Cayley-type conditions have numerators which are polynomials in the caustic parameter $\nu$ whose coefficients are given in terms of the variables $a,b,c$. 

For notational simplicity, in the formulas below we sometimes use  the elementary symmetric polynomials in three variables, $p := a + b + c,$ $q := ab + ac + bc,$ $r := abc$. 

\begin{example}[Period 3]\label{ex:period3}
	The condition $D_2=0$ is equivalent to finding roots of \begin{align*}
	G_3(a,b,c,\nu) &= 3(abc)^2 -2 abc\left(ab + bc + ac \right)\nu + \left( 4 a b c (a+b+c)-(a b+a c+b c)^2\right) \nu^2 \\
		& = 3r^2 - 2qr \nu + (4pr-q^2)\nu^2
	\end{align*} 
in $\nu$, and its discriminant with respect to $\nu$ is 
$$
	\begin{aligned}
	\mathcal{D}_\nu G_3 &= 2^4 (abc)^2 \left(a^2 b^2 + a^2 c^2 + b^2 c^2 - a b c (a + b + c)\right) \\
	&= 2^4 r^2 \left(q^2-3 p r\right).
	\end{aligned}
$$
\end{example}

\begin{example}[Period 4]
	Solving the equation $B_3 =0$ is equivalent to finding roots of $$G_4(a,b,c,\nu) = (\nu  (-a b+a c+b c)-a b c) (\nu  (a b+a c-b c)-a b c) (\nu  (a b-a c+b c)-a b c)$$ in $\nu$, and its discriminant with respect to $\nu$ is 
	\begin{align*}
	\mathcal{D}_\nu G_4 &= 2^6 (abc)^8 (a-b)^2 (a-c)^2 (b-c)^2 \\
	&= 64 r^8 \left(p^2 q^2-4 p^3 r+18 p q r-4 q^3-27 r^2\right).
	\end{align*}
	\label{Period4Example}
\end{example}

\begin{example}[Period 5]
	The formula $D_2D_4 - D_3^2 =0$ is equivalent to finding roots of 
	\begin{align*}
	G_5(a,b,c,\nu) &= 5 r^6  -10 q r^5\nu + r^4 \left(52 p r-9 q^2\right)\nu^2 + 4 r^3 \left(-36 p q r+9 q^3+56 r^2\right) \nu^3 \\
	& \qquad + r^2 \left(-16 r^2 \left(p^2+14 q\right)+120 p q^2 r-29 q^4\right)\nu^4  \\
	& \qquad + 2 r \left(16 q r^2 \left(q-p^2\right)-8 p q^3 r+64 p r^3+3 q^5\right)\nu^5 \\
	& \qquad  + (48 p^2 q^2 r^2-64 r^3 \left(p^3+4 r\right)-12 p q^4 r+128 p q r^3-32 q^3 r^2+q^6)\nu^6
	\end{align*}
	in $\nu$, and its discriminant with respect to $\nu$ is 
	\begin{align*}
	\mathcal{D}_\nu G_5 &= 2^{44}\cdot 5 \cdot r^{38} \left(p^2 q^2-4 p^3 r+18 p q r-4 q^3-27 r^2\right)^4 \\
	& \qquad \times \left(-889 p^2 q^2 r^2+r^3 \left(1369 p^3+4320 r\right)+243 p q^4 r-2880 p q r^3+640 q^3 r^2-27 q^6\right).
	\end{align*}
\end{example}

\begin{example}[Period 6]
	Finding the solutions of $B_3B_5 - B_4^2 =0$ is equivalent to finding roots of 
	\begin{align*}
	G_6(a,b,c,\nu) &=\left[\left(-3 a^2 b^2+c^2 (a-b)^2+2 a b c (a+b)\right)\nu ^2 +2 a b c (a b-ac-bc) \nu+ (abc)^2\right] \\
	&\times \left[(-a^2 (b-c)^2+2 a b c (b+c)-b^2 c^2)\nu^2 - 2abc(ab + ac + bc)\nu + 3(abc)^2\right] \\
	&\times \left[(a^2 (b-c)^2+2 a b c (b+c)-3 b^2 c^2)\nu^2 + 2abc (-ab -ac +bc )\nu + (abc)^2\right] \\
	&\times \left[(a^2 (b-c) (b+3 c)+2 a b c (c-b)+b^2 c^2)\nu^2 + 2abc(-ab + ac - bc)\nu + (abc)^2 \right]
	\end{align*}
	in $\nu$, and its discriminant with respect to $\nu$ is 
	\begin{align*}
	\mathcal{D}_\nu G_6 &= -2^{88}(abc)^{74}(a-b)^{18} (a-c)^{18} (b-c)^{18} \left(a^2 b^2 + a^2 c^2 + b^2 c^2 - a b c (a + b + c)\right) \\
	&= 2^{88} r^{74} \left(q^2 - 3 p r\right) \left(-p^2 q^2+4 p^3 r-18 p q r+4 q^3+27 r^2\right)^9.
	\end{align*}
\end{example}

\begin{example}[Period 7]\label{ex:period7}
	The condition $$\det \threebythree{D_2}{D_3}{D_4}{D_3}{D_4}{D_5}{D_4}{D_5}{D_6}=0$$ is equivalent to finding roots of a polynomial $G_7(a,b,c,\nu)$
	in $\nu$ of degree 12.
	The discriminant of that polynomial with respect to $\nu$ is 
\begin{align*}
	\mathcal{D}_\nu G_7 &= -2^{184}\cdot 7^2 \cdot r^{172} \left(p^2 q^2-4 p^3 r+18 p q r-4 q^3-27 r^2\right)^{20}  \\
	&\times \left[13884993 p^2 q^8 r^2-4 q^6 r^3 \left(19497321 p^3+36960632 r\right)-633232064 p^2 q^5 r^4 \right. \\
	& \qquad + \left. p q^4 r^4 \left(254629897 p^3+1330582752 r\right)+64 q^3 r^5 \left(17805509 p^3-16979328 r\right) \right.\\ 
	& \qquad  -2 p^2 q^2 r^5 \left(209755567 p^3+1588370256 r\right)-576 p q r^6 \left(846895 p^3-8489664 r\right) \\
	& \qquad + r^6 \left(731717280 p^3 r+250406527 p^6-3667534848 r^2\right)+134695872 p q^7 r^3 \\
	& \qquad \left. -1518750 p q^{10} r-9977472 q^9 r^2+84375 q^{12}\right].
\end{align*} 
\end{example}

Each of the polynomials $G_i(a,b,c,\nu)$ from Examples \ref{ex:period3}--\ref{ex:period7} is discriminantly factorizable. But in contrast to the examples in \cite{ADR2019}, there is no obvious variable change that leads to discriminantly separable polynomials for the above examples. However, some are \emph{nearly} discriminantly separable in the variables $a$, $d = b/a$, and $e=c/a$. For example, $$\mathcal{D}_\nu G_4 = 64 a^{30} (d-1)^2 d^8 (e-1)^2 e^8 (d-e)^2,$$ where $(d-e)^2$ is the disqualifying factor. Similar calculations with the same variable change can be made that lead to expressions that are a product of polynomials in the form $$\mathcal{D}_\nu G_i (a,b,c,\nu) = f_1(a)f_2(d)f_3(e)f_4(d,e).$$

Another possible variable change is informed by the similarity of the polynomial $G_3$ and $G_2(a,b,\gamma)$ in \cite{ADR2019}. In terms of the elementary symmetric polynomials $p$, $q$, and $r$, first apply the transformation $(p,q,r) \mapsto (AB, A+B,1)$. This produces $$\mathcal{D}_\nu G_3(A,B) = 2^4((A+B)^2 - 3AB). $$ Applying one more transformation $(A,B) \mapsto(A,C := B/A)$ produces a discriminantly separable polynomial $$\mathcal{D}_\nu G_3(A,C) = 2^4 A^2 (1-C+C^2). $$ However, this double variable change does not produce discriminantly separable polynomials for any of the other examples computed above. For example, this double variable change applied to Example \ref{Period4Example} results in  
$$\mathcal{D}_\nu G_4 = 2^6 (A^6 (-1 + C)^2 C^2 - A^3 (4 - 6 C - 6 C^2 + 4 C^3)-27),$$ which is discriminantly factorizable but not discriminantly separable.

\section{Periodic trajectories and extremal polynomials}
\label{sec:extremal}

\subsection{Polynomial equations as periodicity conditions}
We can formulate the periodicity conditions of Theorem \ref{CayleyThm} in terms of the existence of solutions to certain polynomial equations. 

\begin{theorem}
	The billiard trajectories in the collared and transverse $\Hyp$-ellipses with caustic $\mathcal{C}_\nu$ are $n$-periodic if and only if there exists a pair of polynomials $p_{d_1} = e_{d_1} x^{d_1} + \cdots$ and  $q_{d_2} = f_{d_2} x^{d_2}+ \cdots$ of degrees $d_1$, $d_2$ respectively, with $k = e_{d_1}^2 - \ep f_{d_2}^2$, such that 
	\begin{enumerate}[a)]
		\item if $n=2m$, then $d_1=m$, $d_2 = m-2$, and \\ 
		\begin{equation}\label{PolyEqEven} p_m^2(s) - \left(\frac{1}{a}-s\right) \left(\frac{1}{b}-s\right) \left(\frac{1}{c}-s\right) \left(\frac{1}{\nu}-s\right) q_{m-2}^2(s) = \sign{k}; \end{equation} \\
		\item if $n=2m+1$, then $d_1=m$, $d_2 = m-1$, and \\
		\begin{equation}\label{PolyEqOdd} \left( \frac{1}{\nu}-s \right) p_m^2(s) - \left(\frac{1}{a}-s\right) \left(\frac{1}{b}-s\right) \left(\frac{1}{c}-s\right) q_{m-1}^2(s) = \sign(k\nu). \end{equation}
	\end{enumerate}
\end{theorem}

\begin{proof}
	Note that the proofs of lemma 5.4 and theorem 5.6 in \cite{GaR} together imply the existence of a nontrivial linear combination of the bases for even and odd period $n$ with a zero of order $n$ at $X=0$. 
	
	First consider case $n=2m$. There are real polynomials $p_m^*(X)$ and $q_{m-2}^*(X)$ of degrees $m$ and $m-2$, respectively, such that the expression
	$$ p_m^*(X) - q_{m-2}^*(X)\sqrt{\ep(X-a)(X-b)(X-c)(X-\nu)}$$
	has a zero of order $2m$ at $X=0$. Multiplying this expression by its algebraic conjugate $$p_m^*(X) + q_{m-2}^*(X)\sqrt{\ep(X-a)(X-b)(X-c)(X-\nu)},$$ we arrive at a polynomial of degree $2m$ of the form 
	$$ \left[p_m^*(X)\right]^2 - \ep\left[q_{m-2}^*(X)\right]^2 (X-a)(X-b)(X-c)(X-\nu)$$ which has a zero of order $2m$ at $X=0$. It follows that 
	$$ \left[p_m^*(X)\right]^2 - \ep\left[q_{m-2}^*(X)\right]^2 (X-a)(X-b)(X-c)(X-\nu) = k X^{2m}$$ for some nonzero constant $k$. Using the property that $x = |x|\sign{x}$ and dividing both sides of this equation by $|k| X^{2m}$, we get
	$$ \frac{\left[p_m^*(X)\right]^2}{|k| X^{2m}} - \frac{\ep\left[q_{m-2}^*(X)\right]^2 (X-a)(X-b)(X-c)(X-\nu)}{|k| X^{2m}}=\sign{k}.$$
	Let $s=1/X$ and define $$p_m(s) = \frac{s^m p_m^*(1/s)}{\sqrt{|k|}},\qquad q_{m-2}(s) = \frac{s^{m-2}q_{m-2}^*(1/s)\sqrt{ac|b\nu|}}{\sqrt{|k|}}.$$ Then these polynomials $p_m(s)$, $q_{m-2}(s)$ satisfy equation (\ref{PolyEqEven}), proving part (a) above. 
	
	For the case $n=2m+1$ there are polynomials $p_m^*(X)$, $q_{m-1}^*(X)$ of degree $m$ and $m-1$, respectively, such that the expression 
	$$ p_m^*(X) - q_{m-1}^*(X)\sqrt{\frac{\ep(X-a)(X-b)(X-c)}{(X-\nu)}}$$
	has a zero of order $2m+1$ at $X=0$. Multiplying by
	$$(X-\nu) \left(p_m^*(X) + q_{m-1}^*(X)\sqrt{\frac{\ep(X-a)(X-b)(X-c)}{(X-\nu)}} \right)$$
	we get a polynomial of the form
	$$(X-\nu)\left[p_m^*(X) \right]^2 - \ep\left[q_{m-1}^*(X)\right]^2(X-a)(X-b)(X-c),$$
	which has a zero of order $2m+1$ at $X=0$. Since the degree of this expression is $2m+1$, it follows that
	$$(X-\nu)\left[p_m^*(X) \right]^2 - \ep\left[q_{m-1}^*(X)\right]^2(X-a)(X-b)(X-c) = k X^{2m+1}$$ for some nonzero constant $k$. Again rewriting $k = |k|\sign{k}$ and dividing both sides by $|k| X^{2m+1}$ we get
	$$\frac{(X-\nu)\left[p_m^*(X) \right]^2}{|k| X^{2m+1}} - \frac{\ep\left[q_{m-1}^*(X)\right]^2(X-a)(X-b)(X-c)}{|k| X^{2m+1}} = \sign{k}.$$
	Again let $s=1/X$ and define 
	$$p_m(s) = \frac{s^m p_m^*(1/s)\sqrt{|\nu|}}{\sqrt{|k|}},\qquad q_{m-1}(s) = \frac{s^{m-1}q_{m-1}^*(1/s)\sqrt{a|b|c}}{\sqrt{|k|}}.$$
	Then these polynomials $p_m(s)$ and $q_{m-1}(s)$ satisfy equation (\ref{PolyEqOdd}) above, proving part (b). 
\end{proof}

\begin{corollary}\label{cor:periodicPell}
	If the billiard trajectories inside the collard and transverse $\Hyp$-ellipses are $n$-periodic with caustic $\mathcal{C}_\nu$, then there exist real polynomials $\widehat{p}_n$ and $\widehat{q}_{n-2}$ of degrees $n$ and $n-2$, respectively, which satisfy the Pell equation
	\begin{equation}\label{PellEq}
	\widehat{p}_n(s)^2 - \left(\frac{1}{a}-s\right) \left(\frac{1}{b}-s\right) \left(\frac{1}{c}-s\right) \left(\frac{1}{\nu}-s\right) \widehat{q}_{n-2}(s)^2 =1.
	\end{equation}
\end{corollary}

\begin{proof}
	For $n=2m$, write $\widehat{p}_n = 2p_m^2-\sign{k}$ and $\widehat{q}_{n-2} = 2p_m q_{m-2}$. And for $n=2m+1$, write $\widehat{p}_n = 2\left(\frac{1}{\nu}-s \right) p_m^2-\sign{k\nu}$ and $\widehat{q}_{n-2} = 2p_m q_{m-1}.$
\end{proof}

The Pell-type equations above arise as a functional polynomial condition for periodicity. These solutions of Pell equations have further connections to geometric properties of the billiard trajectories. One can compare these results with \cite{ADR2019}.

\subsection{Rotation numbers}
\label{RotNum}

Suppose $c_0<c_1<c_2<c_3$ are given constants and:
$$
T(s)=(s-c_0)(s-c_1)(s-c_2)(s-c_3).
$$
Then, there exist polynomials $\hat{p}_n$ and $\hat{q}_{n-2}$ of degrees $n$ and $n-2$ respectively such that
$$
\hat{p}_n^2(s)-T(s)\hat{q}_{n-2}^2(s)=1
$$
if and only if there is an integer $n_1>0$ such that
$$
n_1\int_{c_1}^{c_2}\frac{ds}{\sqrt{T(s)}}
=
n\int_{c_3}^{+\infty}\frac{ds}{\sqrt{T(s)}}.
$$
Here $n_1$ is the number of zeroes of $\hat{p}_n$ in $(c_0,c_1)$, see \cite{KLN1990}. Thus, we can define the rotation number as:
$$
\rho
:=
\frac{n_1}{n}
=
\frac{\int_{c_3}^{+\infty}\frac{ds}{\sqrt{T(s)}}}
{\int_{c_1}^{c_2}\frac{ds}{\sqrt{T(s)}}}.
$$
\begin{lemma}\label{lemma:rotation}
	In the above notation, the following relations take place:
	$$0<n_1<n, \quad
	0<\rho<1.
	$$
\end{lemma}
We will consider the case when the boundary is a collared $\mathcal{H}$-ellipse.
In this case, there are three possibilities for types of trajectories.
\begin{itemize}
    \item [(i)] \emph{The caustic is of elliptic type outside of $\mathcal{E}$ and the billiard is within $\mathcal{E}$.} Then $\nu < 0$ and $(\lambda_1,\lambda_2) \in [0,a]\times [b,c]$. 
The condition for $n$-periodicity is:
$$
m_0\int_0^{a}\frac{d\lambda}{\sqrt{\mathcal{P}(\lambda)}}
+
m_1\int_{b}^{c}\frac{d\lambda}{\sqrt{\mathcal{P}(\lambda)}}
=0.
$$
We make the change $s=1/\lambda$, set $c_0=1/\nu$, $c_1=1/c$, $c_2=1/b$,  $c_3=1/a$ and get:
$$
m_0\int_{\infty}^{c_3}\frac{ds}{\sqrt{T(s)}}
+
m_1\int_{c_2}^{c_1}\frac{ds}{\sqrt{T(s)}}
=0.
$$
From
$$
-m_1\int_{c_1}^{c_2}\frac{ds}{\sqrt{T(s)}}=m_0\int_{c_3}^{\infty}\frac{ds}{\sqrt{T(s)}}
$$
and Lemma \ref{lemma:rotation} we conclude:
$$
m_1<0,\, -m_1=n_1<m_0=n,\, \rho = -\frac{m_1}{m_0}.
$$
\item[(ii)] \emph{The caustic is of elliptic type outside of $\mathcal{E}$ and the billiard is outside of $\mathcal{E}$.}
Then in accordance with the discussion in section 3 of \cite{GaR}, the billiard trajectories are space-like and all reflect off of one component of $\mathcal{E}$. Then $\nu<0$ and $(\lambda_1,\lambda_2) \in [\nu, 0]\times[b,c]$. The billiard moves between the one component of $\mathcal{E}$ and the caustic, will not cross the coordinate plane $x_0=0$, but must cross the coordinate planes $x_1=0$ and $x_2=0$ an even number of times. This is the only case which could have an odd period. The condition for $n$-periodicity is:
$$
m_2\int_0^{\nu}\frac{d\lambda}{\sqrt{\mathcal{P}(\lambda)}}
+
m_3\int_{b}^{c}\frac{d\lambda}{\sqrt{\mathcal{P}(\lambda)}}
=0.
$$  
We add and subtract 
$$m_2\int_a^{0}\frac{d\lambda}{\sqrt{\mathcal{P}(\lambda)}}$$
and get
$$
m_2\int_a^{\nu}\frac{d\lambda}{\sqrt{\mathcal{P}(\lambda)}}
+
m_3\int_{b}^{c}\frac{d\lambda}{\sqrt{\mathcal{P}(\lambda)}}
=m_2\int_a^{0}\frac{d\lambda}{\sqrt{\mathcal{P}(\lambda)}}.
$$
Since cycles around $[\nu, a]$ and $[b, c]$ are homologous,
we get
$$
(m_3-m_2)\int_{b}^{c}\frac{d\lambda}{\sqrt{\mathcal{P}(\lambda)}}
=m_2\int_a^{0}\frac{d\lambda}{\sqrt{\mathcal{P}(\lambda)}}.
$$
We make the change $s=1/\lambda$, set $c_0=1/\nu$, $c_1=1/c$, $c_2=1/b$,  $c_3=1/a$ and get:
$$
(m_2-m_3)\int_{c_1}^{c_2}\frac{ds}{\sqrt{T(s)}}=m_2\int_{c_3}^{\infty}\frac{ds}{\sqrt{T(s)}}.
$$
From Lemma \ref{lemma:rotation} we conclude:
$$
m_3<m_2,\, m_2-m_3=n_1<m_2=n,\, \rho = 1-\frac{m_3}{m_2}.
$$
\item[(iii)] \emph{The caustic is of hyperbolic type and the billiard is inside $\mathcal{E}$.}
Then the caustic is symmetric about the plane $x_2=0$ and $\nu \in [b, c]$, so that $(\lambda_1,\lambda_2)\in [0,a]\times[b,\nu]$. The trajectory must become tangent to the caustic at some point inside $\mathcal{E}$. The condition for $n$-periodicity is:
$$
m_4\int_0^{a}\frac{d\lambda}{\sqrt{\mathcal{P}(\lambda)}}
+
m_5\int_{b}^{\nu}\frac{d\lambda}{\sqrt{\mathcal{P}(\lambda)}}
=0.
$$
We make the change $s=1/\lambda$, set $c_0=1/c$, $c_1=1/\nu$, $c_2=1/b$,  $c_3=1/a$ and get:
$$
m_4\int_{\infty}^{c_3}\frac{ds}{\sqrt{T(s)}}
+
m_5\int_{c_2}^{c_1}\frac{ds}{\sqrt{T(s)}}
=0.
$$
From
$$
-m_5\int_{c_1}^{c_2}\frac{ds}{\sqrt{T(s)}}=m_4\int_{c_3}^{\infty}\frac{ds}{\sqrt{T(s)}}
$$
and Lemma \ref{lemma:rotation} we conclude:
$$
m_5<0,\, -m_5=n_1<m_4=n,\, \rho = -\frac{m_5}{m_4}.
$$
\end{itemize}

\subsection{Zolotarev polynomials  and periodic trajectories}

It is well known that the Chebyshev polynomials can be defined recursively  by $T_0(x)=1$, $T_1(x) =x$, and $$T_{n+1}(x) + T_{n-1}(x) = 2x T_n(x)$$ for $n=1,2,\ldots.$ There are other parametrizations like 
\begin{equation}\label{eq:cheb2}
 T_n(x)=\cos n\phi,\quad x=\cos\phi,
\end{equation} 
see e.g.~\cite{AK90}. The reciprocal values of the leading coefficients are $L_n=2^{n-1}$ for $n=1, 2, \dots$ and $L_0=1$. Chebyshev proved that the polynomials $L_nT_n(x)$ are the solutions of the following minmax problem: \emph{find the monic polynomial of degree $n$ which minimizes the uniform norm on the interval $[-1,1]$}. 

As also shown by Chebyshev, the Chebyshev polynomials satisfy the polynomial Pell equation, i.e. there exist
polynomial $Q_{n-1}$ of degree $n-1$ such that
\begin{equation*}
T_n^2(x) - (x-1)(x+1)Q_{n-1}^2=1.
\end{equation*}

Earlier sections show that the Pell equation plays a fundamental role  in a polynomial formulation of periodicity conditions. The solutions of the Pell equation are, up to rescaling,  \emph{the extremal polynomials} in the uniform norm on the union of two intervals defined by the Pell equation. We will call these generalized Chebyshev polynomials on two intervals -- \emph {the Zolotarev polynomials}, since they were introduced in works of Zolotarev \cites{Zolotarev1877}, a prominent student of Chebyshev.  These polynomials were studied further by Akhiezer \cites{Akh1,Akh2,Akh3} (see also \cites{AK47,AK90}). For some most recent results about Zolotarev polynomials see \cite{DS2021}. Following the classics, let us consider the union of two intervals $E_{n,m}=[-1,\alpha_{n,m}]\cup [\beta_{n,m}, 1]$, where
\begin{equation*}
\alpha_{n,m} = 1-2\sn^2\left(\frac{m}{n}K\right), \quad \beta_{n,m} = 2\sn^2\left(\frac{n-m}{n}K\right)-1.
\end{equation*}
Define
\begin{equation*}
TA_{n}(x,m,\kappa)=L\left(v_{n.m}^{n}(u)+\frac{1}{v_{n,m}^{n}(u)}\right),
\end{equation*}
where
$$
v_{n,m}(u)=\dfrac{\theta_1\big(u-\frac{m}{n}K\big)}{\theta_1\big(u+\frac{m}{n}K\big)},
\quad
x_{n,m}=\dfrac{\sn^2(u)\cn^2(\frac{m}{n}K)+\cn^2(u)\sn^2(\frac{m}{n}K)}{\sn^2(u)-\sn^2(\frac{m}{n}K)},
$$
and
$$
L_{n,m}=\frac{1}{2^{n-1}}
\left(\dfrac{\theta_0(0)\theta_3(0)}{\theta_0(\frac{m}{n}K)\theta_3(\frac{m}{n}K)}\right)^{2n},
\quad
\kappa_{n,m}^2=\frac{2(\beta_{n,m} -\alpha_{n,m})}{(1 -\alpha_{n,m})(1+\beta_{n,m})}.
$$
Akhiezer \cites{Akh1,Akh2,Akh3} proved the following result:

\begin{theorem}[Akhiezer]\label{th:Akhiezer}
	\begin{enumerate}[(a)]
		\item The function $TA_{n}(x,m,\kappa)$ is a polynomial  of degree $n$ in $x$ with the leading coefficient $1$ and the second coefficient equal to $-n\tau_{1}^{(n,m)}$, where
		$$
		\tau_{1}^{(n,m)}=-1+2\dfrac{\sn(\frac{m}{n}K)\cn(\frac{m}{n}K)}{\dn(\frac{m}{n}K)}\left(\frac{1}{sn(\frac{2m}{n}K)}-\frac{\theta^{\prime}(\frac{m}{n}K)}{\theta(\frac{m}{n}K)}\right).
		$$
		\item The maximum of the modulus of $TA_{n}$ on the union of the two intervals $[-1,\alpha_{n,m}]\cup [\beta_{n,m}, 1]$ is $L_{n,m}$.
		\item The function $TA_{n}$ takes values $\pm L_{n,m}$ with alternating signs at $\mu=n-m+1$ consecutive points of the interval $ [- 1, \alpha]$ and at $\nu=m+1$ consecutive points of the interval $ [\beta, 1]$. In addition
		$$
		TA_{n}(\alpha_{n,m},m,\kappa_{n,m})=TA_{n}(\beta_{n,m},m,\kappa_{n,m})=(-1)^{m}L_{n,m},
		$$
		and for any $x\in (\alpha_{n,m}, \beta_{n,m})$, it holds:
		$$
		(-1)^{m}TA_{n}(x,m,\kappa_{n,m})>L_{n,m}.
		$$
		\item The polynomials $TA_{n}(x,m,\kappa_{n,m})$ are the Zolotarev  polynomials for $E_{n, m}= [-1,\alpha_{n,m}]\cup [\beta_{n,m}, 1]$ with the norm $L_{n,m}=||TA_{n}(x,m,\kappa_{n,m})||_{E_{n,m}}$ and
		$$E_{n,m}=TA_{n}^{-1}[-L_{n,m}, L_{n,m}].$$
		\item Outside $E_{n,m}$ the derivative of the polynomial $TA_{n}(x,m,\kappa_{n,m})$ with respect to $x$ has only one zero $c_{n,m}$. It belongs to $[\alpha_{n,m}, \beta_{n,m}]$ and
		\begin{equation}\label{eq:cnm}
		c_{n,m}=\frac{\alpha_{n,m}+\beta_{n,m}}{2}-\tau_1^{(n,m)}.
		\end{equation}
		\item Let $F$ be a polynomial of degree $n$ in $x$ with the leading coefficient $1$, such that:
		
		\begin{enumerate}[i)]
			\item $max|F(x)|=L_{n,m}$ for $x\in [-1,\alpha_{n,m}]\cup [\beta_{n,m}, 1]$;
			\item $F(x)$ takes values $\pm L_{n,m}$ with alternating signs at $n-m+1$ consecutive points of the interval $[-1, \alpha_{n,m}]$ and at $m+1$ consecutive points of the interval $[\beta_{n,m}, 1]$.
		\end{enumerate}
		Then $F(x)=TA_{n}(x,m,\kappa_{n,m}).$
	\end{enumerate}
\end{theorem}

The above formulae for $TA_n$ and $E_{n,m}$ provide a complete parametrizations for the Zolotarev polynomials and their supports in the case of two intervals. They can be used in our study of periodic trajectories. We will consider as an example one of the cases of $3$-periodic trajectories, but the same consideration can be applied to any case of any period.

Consider the transversal case of $3$-periodic trajectories
when $b<\nu<0<a<c$. For  $c_0<c_1<c_2<c_3$ we get 
$1/\nu<1/b<1/c<1/a$. We want to construct an affine transformation
$$
h(s)=\hat ls+\hat m: E_{3, m}= [-1,\alpha_{3,m}]\cup [\beta_{3,m}, 1]\xrightarrow[]{} [c_0, c_1]\cup[c_2, c_3].
$$
One of the questions is to determine $m$.

Let us denote $Y=\sn(K/3)$. We are going to calculate $\sn(2K/3)$ in two different ways. The first is by expressing 
$\sn(K-u)$ in terms of $\sn(u), \cn(u), \dn(u)$. The second is in expressing $\sn(2\cdot u)$ in terms of $\sn(u), \cn(u), \dn(u)$. We get
$$
\sn\left(\frac{2K}{3}\right)=\frac{\cn(\frac{K}{3})}{\dn(\frac{K}{3})}
$$
and 
$$
\sn\left(\frac{2K}{3}\right)=\frac{2\sn(\frac{K}{3})\cn(\frac{K}{3})\dn(\frac{K}{3})}{1-\kappa^2\sn^4(\frac{K}{3})}.
$$
By taking the squares and using the formulae to express
$\cn^2(u)$ and $\dn^2(u)$ in terms of $\sn(u)$ and $\kappa$
we get
$$
\kappa^2=\frac{2Y-1}{Y^3(2-Y)}
$$
and 
$$
\sn\left(\frac{2K}{3}\right)=Y(2-Y).
$$
We want to calculate the affine transformation:
$$
h(s)=\hat ls+\hat m: E_{3, m}= [-1,\alpha_{3,m}]\cup [\beta_{3,m}, 1]\xrightarrow[]{} [c_0, c_1]\cup[c_2, c_3].
$$
From $\hat l+\hat m=1/a$, $\hat l\beta+\hat m=1/b$, and $\hat l\alpha +\hat m = 1/b$ we get
\begin{equation}\label{eq:affine1}
\frac{a-\beta c}{c(1-\beta}=\frac{a-\alpha b}{1-\alpha}.
\end{equation}

We have two potential cases: (a) $m=1$ and (b) $m=2$.

\begin{enumerate}[(a)]
	\item $m=1$. Then $\alpha_{3,1}=1-2Y^2$ and $\beta_{3,1}=-1+4Y-2Y^2$. The equation \eqref{eq:affine1}
	leads to 
	\begin{equation}\label{eq:affine2}
	(a-b)c-2(a-b)cY+(bc+ac-ab)Y^2=0.
	\end{equation}
	Uisng $-l+m=1/\nu$ we also get
	\begin{equation}\label{eq:ni1}
	\nu = \frac{ab Y^2}{a-b+bY^2}
	\end{equation}
	However, the last two equations \eqref{eq:affine2} and \eqref{eq:ni1} are not compatible with the equation for the caustic in $3$-periodic case
	\begin{equation}\label{eq:ni3periodic}
	3(abc)^2-2(abc)(ab+bc+ac)\nu +(4abc(a+b+c)-(ab+ac+bc)^2)\nu^2=0.
	\end{equation}
	\item $m=2$. Then $\beta_{3,2}=1-2Y^2$ and $\alpha_{3,2}=1-4Y+2Y^2$. The equation \eqref{eq:affine1}
	leads to 
	\begin{equation}\label{eq:affine3}
	(a-b)c-2b(c-a)Y+a(b-c)Y^2=0.
	\end{equation}
	Using $-\hat l+\hat m=1/\nu$ we also get
	\begin{equation}\label{eq:ni2}
	\nu = \frac{ab Y(2-Y)}{a-b(1-2Y+Y^2)}
	\end{equation}
	The last two equations \eqref{eq:affine3} and \eqref{eq:ni2} are compatible with the equation \eqref{eq:ni3periodic} for the caustic in $3$-periodic case, which here takes the form:
	$$\frac{PQ}{R}=0$$
	with $P=(a-b)c-2b(c-a)Y+a(b-c)Y^2$, $Q=a^2b^2(3c(a-b)+2(ab-2ac-bc)Y -a(b-c)Y^2)$, and $R=(a-b+2bY-bY^2)^2$.
\end{enumerate}
We have proved the following
\begin{proposition}\label{prop:Zolotarev} In the transverse $3$-periodic case with $b<\nu<0<a<c$ the following relations have place, with $Y=\sn(K/3)$:
	$$m=2, \kappa^2=\frac{2Y-1}{Y^3(2-Y)}$$
	$$\hat l=\frac{1}{2c(Y^2-1)}, \hat m=\frac{a+c-2cY^2}{2ac(1-Y^2)}$$
	$$\beta_{3,2}=1-2Y^2, \quad \alpha_{3,2}=1-4Y+2Y^2$$
	and
	$$
	\hat p_3(x)\sim  TZ_3\left(\frac{x-\hat m}{\hat l}; 2; \kappa\right).
	$$
\end{proposition}

Here $\sim$ denotes that two polynomials are equal up to a scalar factor.
The polynomial $\hat p_3$ which appears in Proposition \ref{prop:Zolotarev} is presented in the top part of Figure \ref{fig:Zolotarev}. The polynomial presented on the bottom of Figure  \ref{fig:Zolotarev} cannot be materialized in the case under the consideration. This contrasts the situation in the Euclidean plane where the situation is exactly opposite, see \cite{DR2019a}. 
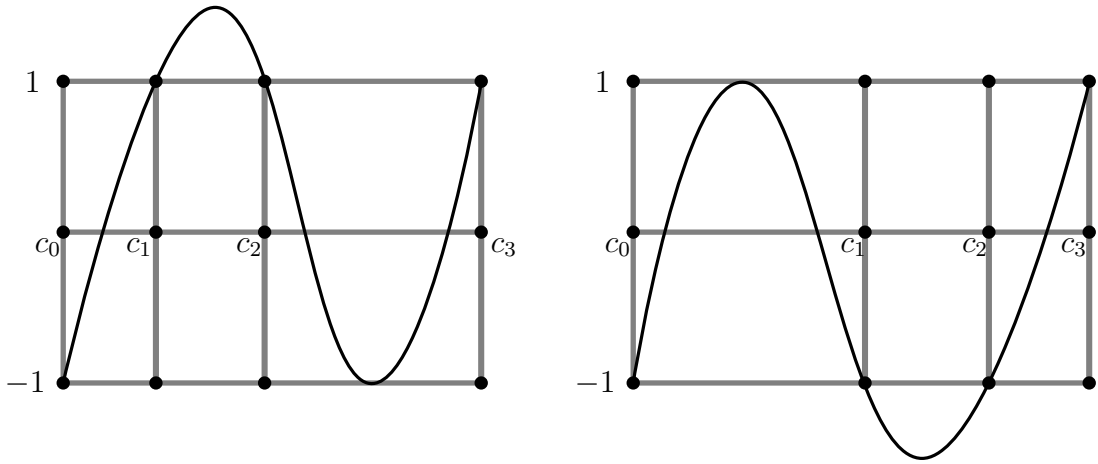
\begin{figure}[htbp]
	\centering
		\begin{tikzpicture}[scale=1]
		
		\draw[line width=.7mm, gray](0,2)--(5.5,2);
		\draw[line width=.7mm, gray](0,0)--(5.5,0);
		\draw[line width=.7mm, gray](0,-2)--(5.5,-2);
		\draw[line width=.7mm, gray](0,2)--(0,-2);
		\draw[line width=.8mm, gray](1.22,-2)--(1.22,2);
		\draw[line width=.8mm, gray](2.65,-2)--(2.65,2);
		\draw[line width=.8mm, gray](5.5,-2)--(5.5,2);
		\draw[gray](0,-2)--(5.5,-2);
		
		
		\draw [very thick] plot [smooth, tension=.8] coordinates { (0,-2) (2,2.98) (4,-2) (5.5,1.98)};
		\fill[black] (2.65,-2) circle (2.5pt);
		\fill[black] (2.65,0) circle (2.5pt);
		\fill[black] (2.65,2) circle (2.5pt);
		\fill[black] (1.22,2) circle (2.5pt);
		\fill[black] (1.22,0) circle (2.5pt);
		\fill[black] (1.22,-2) circle (2.5pt);
		
		\fill[black] (0, 0) circle (2.5pt);
		\fill[black] (0,2) circle (2.5pt);
		\fill[black] (0,-2) circle (2.5pt);
		
		\fill[black] (5.5, 0) circle (2.5pt);
		\fill[black] (5.5,2) circle (2.5pt);
		\fill[black] (5.5,-2) circle (2.5pt);
		
		\draw (-.4,2) node {$1$};
		\draw (-.5,-2) node {$-1$};
		\draw (-.2,-.2) node {${c}_{0}$};
		\draw (1,-.2) node {${c}_{1}$};
		\draw (2.45,-.2) node {${c}_{2}$};
		\draw (5.8,-.2) node {${c}_{3}$};
\begin{scope}[shift={(7.5,0)}]
		\draw[line width=.7mm, gray](0,2)--(6,2);
\draw[line width=.7mm, gray](0,0)--(6,0);
\draw[line width=.7mm, gray](0,-2)--(6,-2);
\draw[line width=.7mm, gray](0,2)--(0,-2);
\draw[line width=.8mm, gray](6,-2)--(6,2);
\draw[line width=.8mm, gray](3.05,-2)--(3.05,2);
\draw[line width=.8mm, gray](4.68,-2)--(4.68,2);
\draw[gray](0,-2)--(6,-2);


\draw [very thick] plot [smooth, tension=.8] coordinates { (0,-2) (1.5,1.98) (3.8,-3) (6,1.98)};

\fill[black] (3.05,-2) circle (2.5pt);
\fill[black] (3.05,0) circle (2.5pt);
\fill[black] (3.05,2) circle (2.5pt);

\fill[black] (4.68,-2) circle (2.5pt);
\fill[black] (4.68,0) circle (2.5pt);
\fill[black] (4.68,2) circle (2.5pt);

\fill[black] (0, 0) circle (2.5pt);
\fill[black] (0,2) circle (2.5pt);
\fill[black] (0,-2) circle (2.5pt);

\fill[black] (6, 0) circle (2.5pt);
\fill[black] (6,2) circle (2.5pt);
\fill[black] (6,-2) circle (2.5pt);

\draw (-.4,2) node {$1$};
\draw (-.5,-2) node {$-1$};
\draw (-.2,-.2) node {${c}_{0}$};
\draw (2.9,-.2) node {${c}_{1}$};
\draw (4.5,-.2) node {${c}_{2}$};
\draw (5.8,-.2) node {${c}_{3}$};

\end{scope}
		
		\end{tikzpicture}
	
	\caption{On the left: the polynomial $\hat{p}_{3}$ corresponding to $n=3$, $m=2$. On the right: the polynomial $\hat{p}_{3}$ corresponding to $n=3$, $m=1$.
	}
	\label{fig:Zolotarev}
\end{figure}

\subsection{Periodic light-like trajectories and Akhiezer polynomials on two symmetric intervals}
By definition, light-like trajectories have velocity $v$ satisfying $\ip{v}{v}=0$ and their caustic is the caustic at infinity, $\mathcal{C}_\infty$. 
As noted in Theorem \ref{th:cayley}, closed light-like trajectories can only be of even period. To that end, we can adjust the above polynomial-based results in the setting of light-like trajectories by considering the limit as $\nu \to \infty$.

\begin{proposition}
	A light-like trajectory in the collared or transverse $\Hyp$-ellipse is periodic with period $n=2m$ if and only if there exist real polynomials $\widehat{p}_n$ and $\widehat{q}_{n-2}$ of degrees $n$ and $n-2$, respectively, which satisfy the Pell equation
	\begin{equation*}
	\widehat{p}_n(s)^2 -s \left(\frac{1}{a}-s\right) \left(\frac{1}{b}-s\right) \left(\frac{1}{c}-s\right) \widehat{q}_{n-2}(s)^2 =1.
	\end{equation*}
\end{proposition}

The Pell equation from the above proposition describes extremal polynomials on two intervals, $[0,1/c] \cup [1/b,1/a]$ or $[1/b,0]\cup[1/c,1/a]$, for the collared or transverse $\Hyp$-ellipse, respectively. We can express each of these in terms of Akhiezer polynomials of even degree composed with an affine transformation, $[c_1,c_2]\cup[c_3,c_4] \mapsto [-1,-\alpha]\cup[\alpha,1]$ for $0 < \alpha < 1$, a simplification of the Zolotarev polynomials previously discussed. In such a case, the polynomails are called \emph{the Akhiezer polynomials}, denoted as  $A_{2m}$ and obtained by a quadratic substitution from the Chebyshev polynomial $T_m$:
\begin{equation}\label{EvenAK}
A_{2m}(x;\alpha) = \frac{(1-\alpha^2)^m}{2^{2m-1}}T_m\left( \frac{2x^2-1-\alpha^2}{1-\alpha^2} \right).
\end{equation}

We illustrate this idea in the example of light-like trajectories of period 4. As noted in \cite{GaR}, the collared $\Hyp$-ellipse can have a light-like period 4 orbit when $c = ab/(b-a)$ and $a<b<2a$; the transverse $\Hyp$-ellipse can have a light-like period 4 orbit when $b = ac/(a-c)$. Under these restrictions on $a$, $b$, $c$ in each case, we can produce a functional solution in terms of Zolotarev polynomials in terms of only two of the parameters $a$, $b$, $c$. 

\begin{proposition}
	Consider a light-like period $4$ trajectory in the collared $\Hyp$-ellipse. The polynomial $\hat{p}_4$ is equal to, up to constant factor, $$\hat{p}_4(s) \sim T_2\left(\frac{2ab^2s^2 - 2b^2s+b-a}{b-a}\right)$$
	where $T_2(x) = 2x^2-1$ and $x=2as-1$. 
\end{proposition}

\begin{proof}
	First we seek to find an affine transformation $$g: [-1,-\alpha] \cup [\alpha,1] \mapsto [0,1/c]\cup[1/b,1/a].$$ Writing $g(x) = Ax+B$, we see that
	$$-A+B = g(-1) = 0, \qquad A+B = g(1) = \frac{1}{a}.$$
	Solving this system implies $A=B=\frac{1}{2a}$. The other two endpoints produce equations 
	$$-A\alpha +B = g(-\alpha) = \frac{1}{c}, \qquad A\alpha + B = g(\alpha) = \frac{1}{b}.$$ 
	Solving each equation gives $\alpha = \dfrac{2a}{b}-1$ and $\alpha = -\dfrac{2a}{c}+1$, which are equal due to the assumption that $c=\dfrac{ab}{b-a}$. Set $x=g^{-1}(s)$. A simple calculation of the composition of $g$ and (\ref{EvenAK}) proves the proposition.
\end{proof}

Repeating the above proof in the case of the transverse $\Hyp$-ellipse produces a similar result. 

\begin{proposition}
	Consider a light-like period $4$ trajectory in the transverse $\Hyp$-ellipse. The polynomial $\hat{p}_4$ is equal to, up to constant factor, $$\hat{p}_4(s) \sim T_2\left(\frac{2a^2cs^2 - 2a^2s-(c-a)}{c-a}\right)$$
	where $T_2(x) = 2x^2-1$ and $x=\dfrac{2acs-a}{2c-a}$. 
\end{proposition}

\subsection{Degenerate cases and classical Chebyshev polynomials}

In this section, we consider the cases the caustic $\mathcal{C}_{\nu}$ is degenerate, in particular  $a<\nu=b<c$.
We will derive the conditions for $2m$-periodicity of the corresponding trajectories.

Note that the discussion from Section \ref{TopProperties} implies that there are two closed trajectories on that level set, both being $2$-periodic.
However, in the limit $\nu\to b$, the rotation number can approach another value.
The next proposition gives the condition for resonance in that limit.

\begin{proposition}\label{prop:ll}
	A trajectory is periodic with period $n=2m$ in the case $a<\nu=b<c$ if and only if there exist real polynomials $\hat{p}_m(s)$ and $\hat{q}_{m-1}(s)$ of degrees $m$ and $m-1$ respectively if and only if:
	\begin{enumerate}[(a)]
		\item $\hat{p}_m^2(s)-\left(s-\dfrac1a\right)\left(s-\dfrac1c\right)\hat{q}_{m-1}^2(s)=1$; and
		\item $\hat{q}_{m-1}(1/b)=0$.
	\end{enumerate}	
\end{proposition}

The first condition from Proposition \ref{prop:ll} is the standard Pell equation describing extremal polynomials on one interval $[1/c,1/a]$, thus the polynomials $\hat{p}_m$ can be obtained as Chebyshev polynomials composed with an affine transformation $[1/c,1/a]\to[-1,1]$.
The additional condition $\hat{q}_{m-1}(1/b)=0$ implies an additional constraint on parameters $a$, $b$ and $c$.
We have the following

\begin{proposition} The polynomials $\hat {p}_m$ and the parameters $a, b, c$ have the following properties:
	\begin{enumerate}[(a)]
		\item
		$\hat{p}_m(s)
		=
		T_m
		\left(
		\dfrac{2ac}{c-a}
		\left(s-\dfrac{a+c}{2ac}\right)\right)$, 
		where $T_m$ is defined by \eqref{eq:cheb2};
		\item
		the condition $\hat{q}_{m-1}(1/b)=0$ is equivalent to 
		$$
		x_0=\cos \left(\frac{k}{m}\pi\right), \quad k=1, \dots, m-1,
		$$
		for
		$$
		x_0=\frac{2ac-b(c+a)}{c-a}.
		$$
	\end{enumerate}
\end{proposition}
\begin{proof}
	The increasing affine transformation $h:[-1, 1]\rightarrow [1/c, 1/a]$ is given by the formula $h(s)=\hat ls + \hat m$,
	where 
	$$
	\hat m = \frac{a+c}{2ac}, \quad \hat l = \frac{c-a}{2ac}.
	$$
	We apply Corollary \ref{cor:periodicPell}.
	The internal extremal points of the Chebyshev polynomial $T_m$ of degree $m$ on the interval $[-1, 1]$ are given by
	$$
	x_k=\cos \left(\frac{k}{m}\pi\right), \quad k=1, \dots, m-1,
	$$
	according to the formula \eqref{eq:cheb2}. The second item follows from $h^{-1}(1/b)=x_k$. 
\end{proof}

\section*{Acknowledgements}

This research is supported by the Discovery Project No.~DP190101838 \emph{Billiards within confocal quadrics and beyond} from the Australian Research Council.
Research of V.~D. and M.~R. is also supported by Mathematical Institute SANU, the Ministry of Education, Science and Technological Development of the Republic of Serbia, and the Science Fund of Serbia.

\bibliographystyle{amsalpha}
\nocite{*}
\bibliography{References1}

\end{document}